\documentclass[10pt]{amsart}
\usepackage{amsrefs}
\usepackage{amssymb}
\usepackage[all]{xy}

\usepackage{hyperref}

\usepackage{tikz}
\usetikzlibrary{matrix,arrows}

\DeclareMathOperator{\Gal}{Gal}

\DeclareMathOperator{\Img}{Im}

\DeclareMathOperator{\Ker}{Ker}

\DeclareFontFamily{U}{wncy}{}
\DeclareFontShape{U}{wncy}{m}{n}{<->wncyr10}{}
\DeclareSymbolFont{mcy}{U}{wncy}{m}{n}
\DeclareMathSymbol{\Sha}{\mathord}{mcy}{"58}
\DeclareMathSymbol{\sha}{\mathord}{mcy}{"78}

\begin{document}

\newtheorem{thm}{Theorem}[section]
\newtheorem{cor}[thm]{Corollary}
\newtheorem{lem}[thm]{Lemma}
\newtheorem{fact}[thm]{Fact}
\newtheorem{prop}[thm]{Proposition}
\newtheorem{defin}[thm]{Definition}
\newtheorem{exam}[thm]{Example}
\newtheorem{examples}[thm]{Examples}
\newtheorem{rem}[thm]{Remark}
\newtheorem{case}{\sl Case}
\newtheorem{claim}{Claim}
\newtheorem{question}[thm]{Question}
\newtheorem{conj}[thm]{Conjecture}
\newtheorem*{notation}{Notation}
\swapnumbers
\newtheorem{rems}[thm]{Remarks}
\newtheorem*{acknowledgment}{Acknowledgements}
\newtheorem*{thmno}{Theorem}

\newtheorem{questions}[thm]{Questions}
\numberwithin{equation}{section}

\newcommand{\gr}{\mathrm{gr}}
\newcommand{\inv}{^{-1}}
\newcommand{\isom}{\cong}
\newcommand{\dbC}{\mathbb{C}}
\newcommand{\F}{\mathbb{F}}
\newcommand{\dbN}{\mathbb{N}}
\newcommand{\Q}{\mathbb{Q}}
\newcommand{\dbR}{\mathbb{R}}
\newcommand{\dbU}{\mathbb{U}}
\newcommand{\Z}{\mathbb{Z}}
\newcommand{\calG}{\mathcal{G}}
\newcommand{\K}{\mathbb{K}}
\newcommand{\rmH}{\mathrm{H}}
\newcommand{\bfH}{\mathbf{H}}
\newcommand{\rmr}{\mathrm{r}}
\newcommand{\Span}{\mathrm{Span}}
\newcommand{\eue}{\mathbf{e}}

%%%%%%%%%%%%%%%%%
% New commands for our things - such as cocycles and so on

\newcommand{\hac}{\hat c}
\newcommand{\hatheta}{\hat\theta}
%%%%%%%%%%%%%%%%%%%%%%%%%%%%%%%%%%

\title[Massey products and the elementary type conjecture]{Massey products in Galois cohomology \\ and the elementary type conjecture}

\author{Claudio Quadrelli}
\address{Department of Science \& High-Tech, University of Insubria, Como, Italy EU}
\email{claudio.quadrelli@uninsubria.it}
\date{\today}
%\thanks{W la pastasciutta!}
%\dedicatory{To Pablo Spiga, an enthusiast algebraist \\ and a ``graphomaniac'', with admiration.}

\begin{abstract}
 Let $p$ be a prime. 
 We prove that a positive solution to Efrat's Elementary Type Conjecture implies a positive solution to a strengthened version of Mina\v{c}--T\^an's Massey Vanishing Conjecture in the case of finitely generated maximal pro-$p$ Galois groups whose pro-$p$ cyclotomic character has torsion-free image.
 Consequently, the maximal pro-$p$ Galois group of a field $\K$ containing a root of 1 of order $p$ (and also $\sqrt{-1}$ if $p=2$) satisfies the strong $n$-Massey vanishing property for every $n>2$ (which is equivalent to the cup-defining $n$-Massey product property for every $n>2$, as defined by Mina\v{c}--T\^an) in several relevant cases.
\end{abstract}



\maketitle
%%%%%%%%%%%%%%%%%%%%%%%%%%%%%%%%%%%%%%%%%%%%%%%%%

\section{Introduction}
\label{sec:intro}

Throughout the paper, $p$ will denote a prime number.
Given a field $\K$, let $\bar\K_s$ denote a separable closure of $\K$, and let $\K(p)$ denote the maximal pro-$p$-extension of $\K$ inside $\bar\K_s$.
The {\sl absolute Galois group} $G_{\K}=\Gal(\bar\K_s/\K)$ is a profinite group, and the Galois group 
$$G_{\K}(p):=\Gal(\K(p)/\K)=\varprojlim_{[\mathbb{L}:\K]=p^k}\Gal(\mathbb{L}/\K),$$
called the {\sl maximal pro-$p$ Galois group} of $\K$, is the maximal pro-$p$ quotient of $G_{\K}$.
A major difficult problem in Galois theory is the characterization of profinite groups which occur as absolute Galois groups of fields, and of pro-$p$ groups which occur as maximal pro-$p$ Galois groups (see, e.g., \cite[\S~3.12]{WDG} and \cite[\S~2.2]{birs}).
Observe that if a pro-$p$ group $G$ does not occur as the maximal pro-$p$ Galois group of a field containing a root of 1 of order $p$, then it does not occur as the absolute Galois group of any field.

In the '90s, I.~Efrat formulated a conjecture --- the {\sl Elementary Type Conjecture} on maximal pro-$p$ Galois groups, see \cite{ido:etc} --- which proposes a description of finitely generated pro-$p$ groups which occur as maximal pro-$p$ Galois groups containing a root of 1 of order $p$: it predicts that if $\K$ is a field containing a root of 1 of order $p$ with $G_{\K}(p)$ finitely generated, then $G_{\K}(p)$ may be constructed starting from free pro-$p$ groups and Demushkin pro-$p$ groups and iterating free pro-$p$ products and certain semidirect products with $\Z_p$ (see also \cite[\S~10]{marshall} and \cite[\S~7.5]{qw:cyc}). 
The pro-$p$ groups which are constructible in this way are called {\sl pro-$p$ groups of elementary type} (see Definition~\ref{defin:ET} below).
The Elementary Type Conjecture is verified, for example, if
$\K$ is %a pseudo algebraically closed (PAC) field or 
an extension of relative transcendence degree 1 of a pseudo algebraically closed field (see \cite[Ch.~11]{friedjarden} and \cite[\S~5]{ido:Hasse}); or if $\K$ is an algebraic extension of a global field of characteristic not $p$  (see \cite{ido:etc2}).
Moreover, in \cite{sz:raags} I.~Snopce and P.A.~Zalesski{\u\i} provided new evidence in support of the Elementary Type Conjecture, as they proved that within the family of {\sl right-angled Artin pro-$p$ groups} --- which is an extremely rich family of pro-$p$ groups ---, the only members which occur as maximal pro-$p$ Galois groups (and thus as absolute Galois groups) are of elementary type.% (recently, S.~Blumer, Th.~Weigel, and the author, proved the same result for a wider family of pro-$p$ groups, i.e., the family of {\sl oriented right-angled Artin pro-$p$ groups} see \cite{BQW}).

The proof of the celebrated {\sl Bloch-Kato conjecture} --- now called {\sl Norm Residue Theorem} --- by M.~Rost and V.~Voevodsky, with the so-called ``Weibel's patch'' (see \cite{rost,voev,weibel,HW:book}) provided new insights in the study of maximal pro-$p$ Galois groups and absolute Galois groups of fields (see, e.g., \cite{cem,idojan:tams,MPQT} and references therein).
Indeed, the Norm Residue Theorem implies that the ring structure of the $\F_p$-cohomology algebra
\[
 \bfH^\bullet(G_{\K}(p))=\coprod_{n\geq0}\rmH^n(G_{\K}(p),\F_p)
\]
of a field $\K$ containing a root of 1 of order $p$, endowed with the graded-commutative {\sl cup-product} 
$$\textvisiblespace\smallsmile\textvisiblespace\colon \rmH^s(G,\F_p)\times\rmH^t(G,\F_p)\longrightarrow\rmH^{s+t}(G,\F_p),
\qquad s,t\geq0,$$
is determined by degrees 1 and 2 (see, e.g., \cite[\S~1]{cq:bk}) --- 
observe that this is true also for every closed subgroup of $G_{\K}(p)$, as every closed subgroup is again a maximal pro-$p$ Galois group.
It is worth underlining that all pro-$p$ groups of elementary type satisfy hereditarily this cohomological condition (i.e., 
the ring structure of the $\F_p$-cohomology algebra of every closed subgroup is determined by degrees 1 and 2, see \cite[Thm.~1.4]{qw:cyc}); on the other hand, it is remarkable that we do not know examples of finitely generated pro-$p$ groups satisfying hereditarily this cohomological condition other than pro-$p$ groups of elementary type.

In recent years --- especially after the publication of the work of M.~Hopkins and K.~Wickelgren \cite{hopwick} ---, much of the research on absolute Galois groups and maximal pro-$p$ Galois groups focused on the study of {\sl Massey products} in Galois cohomology (see, e.g., \cite{vogel,ido:Massey,mt:Massey,birs} and references therein).
Given a pro-$p$ group $G$ and an integer $n\geq2$, the {\sl $n$-fold Massey product} is a multi-valued map which associates a 
sequence $\alpha_1,\ldots,\alpha_n$ of elements of $\rmH^1(G,\Z/p)$ to a 
(possibly empty) subset $$\langle\alpha_1,\ldots,\alpha_n\rangle\subseteq \rmH^2(G,\F_p).$$ 
If $n=2$ it coincides with the cup-product, namely, $\langle\alpha_1,\alpha_2\rangle=\{\alpha_1\smallsmile\alpha_2\}$.
For $n>2$, a pro-$p$ group $G$ is said to satisfy the {\sl $n$-Massey vanishing property} if the set 
$\langle\alpha_1,\ldots,\alpha_n\rangle$ contains 0 whenever it is non-empty.
In \cite{mt:conj}, J.~Mina\v{c} and N.D.~T\^an conjectured the following.

\begin{conj}\label{conj:MvanMT}
 Let $\K$ be a field containing a root of 1 of order $p$.
 Then the maximal pro-$p$ Galois group $G_{\K}(p)$ of $\K$ satisfies the $n$-Massey vanishing property for every $n>2$.
\end{conj}

One has the following partial --- but very remarkable --- results:
\begin{itemize}
 \item[(a)] E.~Matzri proved that the maximal pro-$p$ Galois group of every field containing a root of 1 of order $p$ satisfies the $3$-Massey vanishing property (see the preprint \cite{eli:Massey}, see also the published works \cite{EM:Massey,MT:Masseyall});
 \item[(b)] J.~Mina\v{c} and N.D.~T\^an proved Conjecture~\ref{conj:MvanMT} for local fields (see \cite{mt:Massey});
 \item[(c)] Y.~Harpaz and O.~Wittenberg proved Conjecture~\ref{conj:MvanMT} for number fields (see \cite{HW:Massey});
 \item[(d)] A.~Merkurjev and F.~Scavia proved that the maximal pro-$2$ Galois group of every field satisfies the $4$-Massey vanishing property (see \cite{MerSca2});
 \item[(e)] A.~P\'al and G.~Quick proved a {\sl strengthened version} of Conjecture~\ref{conj:MvanMT} in case $p=2$ for fields with virtual cohomological dimension at most 1 (see \cite{palquick:Massey}).
\end{itemize}
Further interesting results on Massey products in Galois cohomology have been obtained by various authors (see, e.g., \cite{eli2,eli3,wick,jochen,GPM,PJ,LLSWW,MerSca3}).

The purpose of the present work is to prove a {\sl strengthened version} of Conjecture~\ref{conj:MvanMT} for fields whose maximal pro-$p$ Galois group is of elementary type.

Given a pro-$p$ group $G$ and a positive integer $n>2$, if the set $\langle\alpha_1,\ldots,\alpha_n\rangle$, associated to a sequence $\alpha_1,\ldots,\alpha_n$ of elements of $\rmH^1(G,\F_p)$, is non-empty, then necessarily 
\begin{equation}\label{eq:cup 0}
 \alpha_1\smallsmile\alpha_2=\alpha_2\smallsmile\alpha_3=\ldots=\alpha_{n-1}\smallsmile\alpha_n=0
\end{equation}
--- we underline that the triviality condition \eqref{eq:cup 0} is also sufficient to imply the non-emptiness of $\langle\alpha_1,\alpha_2,\alpha_3\rangle$, in the case $n=3$.
A pro-$p$ group $G$ is said to satisfy the {\sl strong} $n$-Massey vanishing property, for $n>2$, if every sequence $\alpha_1,\ldots,\alpha_n$ of elements of $\rmH^1(G,\F_p)$ satisfying condition \eqref{eq:cup 0} yields an $n$-fold Massey product $\langle\alpha_1,\ldots,\alpha_n\rangle$ containing 0 (see \cite[Def.~1.2]{pal:Massey}). 
The strong $n$-Massey vanishing property is stronger than the $n$-Massey vanishing property; observe that, since the two properties coincide for $n=3$, E.~Matzri's result implies that the maximal pro-$p$ Galois group of a field containing a root of 1 of order $p$ satisfies, ``for free'', also the strong 3-Massey vanishing property.

By construction, a pro-$p$ group of elementary type $G$ comes endowed with an {\sl orientation}, namely, a homomorphism of pro-$p$ groups $G\to1+p\Z_p$, where $1+p\Z_p$ denotes the multiplicative group of principal units of $\Z_p$ (we give a brief review of orientations of pro-$p$ groups in \S~\ref{ssec:or}).
Our main result is the following.

\begin{thm}\label{thm:main}
Let $G$ be a pro-$p$ group of elementary type.
If $p=2$ assume further that the image of the orientation associated to $G$ is 
%either $\{\pm1\}$ or 
a subgroup of $1+4\Z_2$.
Then $G$ satisfies the strong $n$-Massey vanishing property for every $n>2$.
 \end{thm}
 
% (We recall how the canonical orientation of a Demushkin pro-$p$ group is defined in Example~\ref{exam:Demushkin} below.) 
 To prove Theorem~\ref{thm:main}, we exploit a result --- whose original formulation, for discrete groups, is due to W.~Dwyer, see \cite{dwyer} --- which interprets the vanishing of Massey products in the $\F_p$-cohomology of a pro-$p$ group $G$ in terms of  the existence of certain upper unitriangular representations of $G$.
 Moreover, we use the {\sl Kummerian property} --- a formal version of Hilbert~90, introduced in \cite{eq:kummer} ---, which guarantees the vanishing of ``cyclic'' Massey products (see Theorem~\ref{thm:kummer Massey cyc}), and which is enjoyed both by pro-$p$ groups of elementary type and maximal pro-$p$ Galois groups of fields containing a root of 1 of order $p$ (we give a brief review of the Kummerian property in \S~\ref{ssec:kummer}).

 Let $\K$ be a field containing a root of 1 of order $p$. 
 It has been shown that $G_{\K}(p)$ satisfies the strong $n$-Massey vanishing property for every $n>2$, if $p$ is odd and $\K$ is $p$-rigid, by J.~Mina\v{c} and N.D.~T\^an (see \cite[Thm.~8.5]{mt:conj}); and if $\K$ has virtual cohomological dimension at most 1 or it is pseudo $p$-adically closed, by A.~P\'al and E.~Szab\'o (see \cite{pal:Massey}). 
 As a consequence of Theorem~\ref{thm:main}, we obtain the following.
 
\begin{cor}\label{cor:fields}
 Let $\K$ be a field containing a root of 1 of order $p$, such that the quotient $\K^\times/(\K^\times)^p$ is finite. 
If $p=2$ suppose further that 
%either 
$\sqrt{-1}\in\K$.%, or $\K(\sqrt{-1})$ contains all roots of 1 of order a power of 2.
 Then $G_{\K}(p)$ satisfies the strong $n$-Massey vanishing property for every $n>2$ in the following cases:
  \begin{itemize}
  \item[(a)] $\K$ is a local field, or an extension of transcendence degree 1 of a local field;
\item[(b)] $\K$ is a PAC field, or an extension of relative transcendence degree 1 of a PAC field;
\item[(c)] $\K$ is $p$-rigid {\rm (}for the definition of $p$-rigid fields see \cite[p.~722]{ware}{\rm )};
\item[(d)] $\K$ is an algebraic extension of a global field of characteristic not $p$;
\item[(e)] $\K$ is a valued $p$-Henselian field with residue field $\kappa$, and $G_{\kappa}(p)$ satisfies the strong $n$-Massey vanishing property for every $n>2$.
%\item[(f)] $\K$ is a Pythagorean field, if $p=2$.
 \end{itemize}
\end{cor}

In \cite[Question~4.2]{JT:U4}, J.~Mina\v{c} and N.D.~T\^an asked the following.

\begin{question}\label{conj:MvanMT strong}
Let $p$ be a prime, and let $\K$ be a field containing a root of 1 of order $p$.
 Does the maximal pro-$p$ Galois group $G_{\K}(p)$ of $\K$ satisfy the strong $n$-Massey vanishing property for every $n>2$?
\end{question}

The original formulation of \cite[Question~4.2]{JT:U4} involves the {\sl cup-defining $n$-fold Massey product property}, which is equivalent to the strong $n$-Massey vanishing property, if required for all $n\geq3$ --- see \cite[Rem.~4.6]{JT:U4} and Remark~\ref{rem:cupdef} below.
Question~\ref{conj:MvanMT strong} has a negative answer in case $p=2$.
Indeed, in \cite[Example~A.15]{GPM}, O.~Wittenberg produced an example (suggested by Y.~Harpaz) of a number field $\K$ not containing $\sqrt{-1}$ whose maximal pro-$2$ Galois group $G_{\mathbb{K}}(2)$ does not satisfy the strong $4$-Massey vanishing property.
 Moreover, recently A.~Merkurjev and F.~Scavia showed that every field $\K$ has an extension $\mathbb{L}$ whose maximal pro-$2$ Galois group $G_{\mathbb{L}}(2)$ does not satisfy the strong $4$-Massey vanishing property (cf. \cite[Thm.~6.3]{MerSca1}).
 
Wittenberg's example and Merkurjev-Scavia's result involve pro-2 groups that are not finitely generated.
Thus, we ask whether Question~\ref{conj:MvanMT strong} may have a positive answer under the further conditions that the maximal pro-$p$ Galois group is {\sl finitely generated}, and $\sqrt{-1}\in\K$ if $p=2$.

\begin{question}\label{ques:p2}
Let $\K$ be a field containing a root of 1 of order $p$, such that the quotient $\K^\times/(\K^\times)^p$ is finite.
If $p=2$ suppose further that 
%either 
$\sqrt{-1}\in\K$.
%, or $\K(\sqrt{-1})$ contains all roots of 1 of order a power of 2.
Does the maximal pro-$p$ Galois group $G_{\K}(p)$ of $\K$ satisfy the strong $n$-Massey vanishing property for every $n>2$?
\end{question}

By Theorem~\ref{thm:main}, a positive solution of the Elementary Type Conjecture would yield a positive answer to Question~\ref{ques:p2}.

 %%%%%%%%%%%%%%%%%%%%%%%%%%%%%%%%%%%%%%%%%%%%
 
{\small \subsection*{Acknowledgments}
The author wishes to thank I.~Efrat, E.~Matzri, J.~Mina\v{c}, F.W.~Pasini, and N.D.~T\^an, for several inspiring discussions on Massey products in Galois cohomology and pro-$p$ groups of elementary type, which occurred in the past years; A.~P\'al, F.~Scavia, P.~Wake, O.~Wittenberg, and again J.~Mina\v{c} and N.D.~T\^an, for their useful comments on this work. 
Last, but not least, the author is very grateful to the referees for the careful work carried with the manuscript and for the extremely useful comments and suggestions.%, and for inviting the author to consider also the case of fields not containing $\sqrt{-1}$ and with maximal pro-2 Galois group of elementary type.
}

 %%%%%%%%%%%%%%%%%%%%%%%%%%%%%%%%%%%%%%%%%%%%%%%%%%5
 %%%%%%%%%%5
 %%%%%%%%%%%%%%%%%%%%%%%%%%%%%%%%%%%%%%%%%%%%%%%%%%%%%%%
 
 \section{Massey products and pro-$p$ groups}\label{sec:Massey}
 
  Let $G$ be a pro-$p$ group, and let $\F_p$ be the finite field with $p$ elements, considered as a trivial $G$-module.
 For basic notions on pro-$p$ groups and their $\F_p$-cohomology, we refer to \cite[Ch.~I, \S~4]{serre:galc} and to \cite[Ch.~I, and Ch.~III \S~3]{nsw:cohn}.

 Given a pro-$p$ group $G$, and two subsets $S_1,S_2$ of $\rmH^1(G,\F_p)$, we set
 \[ S_1\smallsmile S_2=\left\{\:\alpha_1\smallsmile\alpha_2\:\mid\:\alpha_1\in S_1,\:\alpha_2\in S_2\:\right\}. \]

 %%%%%%%%%%%%%%%%%%%%%%%%
 
 \subsection{Massey products in Galois cohomology}
 Here we give a brief review on Massey products in the Galois cohomology of pro-$p$ groups.
 Throughout the paper, we will be merely concerned with Massey products of elements in the first cohomology group $\rmH^1(G, \F_p)$, whose definition will be recalled here below.
   Our main references are \cite{vogel} and \cite{mt:Massey} --- for a general definition of Massey products on the level of cochains the reader may consult \cite{dwyer,kraines}. 
 
 For $n\in\{1,2,3\}$ let $\mathfrak{C}^n$ denote the $\F_p$-vector spaces of continuous maps
 $$c\colon \underbrace{G\times\cdots \times G}_{n\text{ times}}\to\F_p$$ 
 ($G\times \cdots \times G$ is to be intended as the cartesian product of topological spaces).
 These vector spaces come equipped with homomorphisms $\partial^n\colon\mathfrak{C}^n\to\mathfrak{C}^{n+1}$, $n=1,2$, defined by
 \[
 \begin{split}
    \partial^1(c) (g_1,g_2) &= c(g_1)-c(g_1g_2)+c(g_2),\\
 \partial^2(c') (g_1,g_2,g_3) &= c'(g_1,g_2)-c'(g_1,g_2g_3)+c'(g_1g_2,g_3)-c'(g_2,g_3),
 \end{split}
 \]
 for every $c\in\mathfrak{C}^1$, $c'\in\mathfrak{C}^2$, and $g_1,g_2,g_3\in G$.
We recall that 
\begin{equation}\label{eq:H1}
 \rmH^1(G,\F_p)=\Ker(\partial^1)=\mathrm{Hom}(G,\F_p)
\end{equation}
--- where the latter is the group of homomorphisms of pro-$p$ groups $G\to\F_p$, with $\F_p$ considered as a cyclic group of order $p$ ---, while 
\begin{equation}\label{eq:H2}
\rmH^2(G,\F_p)=\Ker(\partial^2)/\Img(\partial^1). 
\end{equation}

 For $c,c'\in\mathfrak{C}^1$, one defines $c\cdot c'\in\mathfrak{C}^2$ by $(c\cdot c')(g_1,g_2)=c(g_1)\cdot c'(g_2)$ for every $g_1,g_2\in G$. 
 Then $\partial^2(c\cdot c')=\partial^1(c)\cdot c'-c\cdot\partial^1(c')$ (cf. \cite[Prop.~1.4.1]{nsw:cohn}).
Consequently, for $\alpha,\alpha'\in\rmH^1(G,\F_p)$ one has $\alpha\cdot \alpha'\in\Ker(\partial^2)$, so that one defines the cup-product $\alpha\smallsmile\alpha'$ of $\alpha$ and $\alpha'$ to be the class of $\alpha\cdot \alpha'$ in $\rmH^2(G,\F_p)$.
 
 \begin{rem}\label{rem:cup gradcomm}\rm
  For every $\alpha,\alpha'\in\rmH^1(G,\F_p)$, one has $\alpha'\smallsmile\alpha=-\alpha\smallsmile\alpha'$ (cf. \cite[Prop.~1.4.4]{nsw:cohn}).
  In particular, if $p\neq2$ then $\alpha\smallsmile\alpha=0$.
 \end{rem}

For $n\geq2$ let $\alpha_1,\ldots,\alpha_n$ be a sequence of elements of $\rmH^1(G,\F_p)$. 
A collection $\mathfrak{c}=(c_{ij})$, $1\leq i\leq j\leq n$, of elements of $\mathfrak{C}^1$ is called a {\sl defining set} for the $n$-fold Massey product $\langle\alpha_1,\ldots,\alpha_n\rangle$ if the following conditions hold:
\begin{itemize}
 \item[(a)] $c_{ii}=\alpha_i$ for every $i=1,\ldots,n$;
 
 \item[(b)] for every couple $(i,j)$ such that $1\leq i<j\leq n$ and $(i,j)\neq (1,n)$, one has 
 \begin{equation}\label{eq:defset}
 \partial^1(c_{ij})=\sum_{h=1}^{j-1}c_{i,h}\cdot c_{h+1,j}.
\end{equation}
\end{itemize}
Then $\sum_{h=1}^{n-1}c_{1,h}\cdot c_{h+1,n}$ lies in $\Ker(\partial^2)$, and its class in $\rmH^2(G,\F_p)$ is called the {\sl value} of $\mathfrak{c}$.
The $n$-fold Massey product $\langle\alpha_1,\ldots,\alpha_n\rangle$ is the subset of $\rmH^2(G,\F_p)$ consisting of the values of all its defining sets.
Observe that if $n=2$, then 
\begin{equation}\label{eq:2Massey cup}
\langle\alpha_1,\alpha_2\rangle=\left\{\alpha_1\cdot\alpha_2+\Img(\partial^1)\right\}=\{\alpha_1\smallsmile\alpha_2\}.
\end{equation}
 
 \begin{rem}\label{rem:Massey trivialcup}\rm
Let $\alpha_1,\ldots,\alpha_n$ be a sequence of elements of $\rmH^1(G,\F_p)$, $n>2$. 
By \eqref{eq:defset}, the existence of a defining set for the $n$-fold Massey product $\langle\alpha_1,\ldots,\alpha_n\rangle$ implies that $\alpha_i\cdot\alpha_{i+1}\in\Img(\partial^1)$ for every $i=1,\ldots,n-1$, i.e., $\alpha_i\smallsmile\alpha_{i+1}=0$. 

\noindent Moreover, if $n=3$ this condition is also sufficient for the existence of a defining set $\mathfrak{c}=(c_{ij})$ for the 3-fold Massey product $\langle\alpha_1,\alpha_2,\alpha_3\rangle$: indeed, if $\alpha_1\smallsmile\alpha_2=\alpha_2\smallsmile\alpha_3=0$ then there exist $c_{1,2},c_{2,3}\in\mathfrak{C}^1$ such that $\partial^1(c_{1,2})=\alpha_1\cdot\alpha_2$ and $\partial^1(c_{2,3})=\alpha_2\cdot\alpha_3$, and thus $\langle\alpha_1,\alpha_2,\alpha_3\rangle\neq\varnothing$.
 \end{rem}
 
\begin{defin}\rm
Let $G$ be a pro-$p$ group, let $n$ be a positive integer, $n\geq2$, and let $\alpha_1,\ldots,\alpha_n$ be a sequence of elements of $\rmH^1(G,\F_p)$.
\begin{itemize}
 \item[(a)] The $n$-fold Massey product $\langle\alpha_1,\ldots,\alpha_n\rangle$ is said to be {\sl defined} if it is non-empty --- i.e., if there exists at least one defining set.
 \item[(b)] The $n$-fold Massey product $\langle\alpha_1,\ldots,\alpha_n\rangle$ is said to {\sl vanish} if $0\in \langle\alpha_1,\ldots,\alpha_n\rangle$.
 \end{itemize}
 \end{defin}
 
 \begin{defin}\rm
 Let $G$ be a pro-$p$ group, and let $n$ be a positive integer, $n\geq2$.
\begin{itemize}
 \item[(a)] The group $G$ is said to satisfy the {\sl cup-defining $n$-fold Massey property} (with respect to $\F_p$) if every $n$-fold Massey product $\langle\alpha_1,\ldots,\alpha_n\rangle$ in the $\F_p$-cohomology of $G$ is defined whenever
\begin{equation}\label{eq:cond trivial cup def}
 \alpha_i\smallsmile \alpha_{i+1}=0\qquad\text{for every }i=1,\ldots,n-1.
\end{equation}
 \item[(b)] The group $G$ is said to satisfy the {\sl $n$-Massey vanishing property} (with respect to $\F_p$) if every defined $n$-fold Massey product in the $\F_p$-cohomology of $G$ vanishes.
 \item[(c)] The group $G$ is said to satisfy the {\sl strong} $n$-Massey vanishing property (with respect to $\F_p$) if it satisfies both the cup-defining $n$-fold Massey property and the $n$-fold Massey vanishing property; namely, the $n$-fold Massey product $\langle\alpha_1,\ldots,\alpha_n\rangle$ vanishes whenever condition \eqref{eq:cond trivial cup def} is satisfied.
\end{itemize}
\end{defin}

\begin{rem}\label{rem:cupdef}\rm
If a pro-$p$ group $G$ has the cup-defining $n$-fold Massey property, with $n\geq4$, then $G$ has the vanishing $(n-1)$-fold Massey vanishing property, as observed in \cite[Rem.~4.6]{JT:U4}.
Therefore, $G$ has the strong $n$-fold Massey vanishing property for every $n\geq3$ if, and only if, it has the cup-defining $n$-fold Massey product property for every $n\geq3$.
In particular, \cite[Question~4.2]{JT:U4} is equivalent to Question~\ref{conj:MvanMT strong}. 
\end{rem}

Massey products in the $\F_p$-cohomology of pro-$p$ groups enjoy the following properties (cf., e.g., \cite[Prop.~1.2.3--1.2.4]{vogel} and \cite[Rem.~2.2]{mt:Massey}).

 \begin{prop}\label{prop:Massey cup}
 Let $G$ be a pro-$p$ group and let $\alpha_1,\ldots,\alpha_n$ be a sequence of elements of $\rmH^1(G,\F_p)$, with $n>2$.
 Suppose that the $n$-fold Massey product $\langle\alpha_1,\ldots,\alpha_n\rangle$ is defined.
 \begin{itemize}
 \item[(i)] If $\alpha_i=0$ for some $i$, then the $n$-fold Massey product $\langle\alpha_1,\ldots,\alpha_n\rangle$ vanishes.
 \item[(ii)] If the $n$-fold Massey product $\langle\alpha_1,\ldots,\alpha_n\rangle$ is defined, then for every $a\in\F_p$ and $i\in\{1,\ldots,n\}$ one has 
 \[
\langle\alpha_1,\ldots,a\alpha_i,\ldots,\alpha_n\rangle\supseteq \{\:a\beta\:\mid\:\beta\in\langle\alpha_1,\ldots,\alpha_n\rangle \:\}.
 \] 
\item[(iii)] If the set $\langle\alpha_1,\ldots,\alpha_n\rangle$ is not empty, then it is closed under adding $\alpha_1\smallsmile\alpha'$ and $\alpha_n\smallsmile\alpha'$ for any $\alpha'\in\rmH^1(G,\F_p)$.
\end{itemize}
\end{prop}

 %%%%%%%%%%%%%%%%%%%%%%%%%%%%5
 
 \subsection{Massey products and upper unitriangular matrices}
 
 Massey products may be ``translated'' in terms of unipotent upper-triangular homomorphisms of $G$ as follows.
For $n\geq 2$ let
\[
 \dbU_{n+1}=\left\{\left(\begin{array}{ccccc} 1 & a_{1,2} & \cdots & & a_{1,n+1} \\ & 1 & a_{2,3} &  \cdots & \\
 &&\ddots &\ddots& \vdots \\ &&&1& a_{n,n+1} \\ &&&&1 \end{array}\right)\mid a_{i,j}\in\F_p \right\}\subseteq 
 \mathrm{GL}_{n+1}(\F_p)
\]
be the group of unipotent upper-triangular $(n+1)\times(n+1)$-matrices over $\F_p$.
Let $I_{n+1}$ denote the $(n+1)\times(n+1)$ identity matrix, and for $1\leq i<j\leq n+1$, let $E_{ij}$ denote the $(n+1)\times(n+1)$-matrix with 1 at the entry $(i,j)$, and 0 elsewhere.
We set $\bar \dbU_{n+1}=\dbU_{n+1}/Z$, where $Z$ denotes the normal subgroup
 \begin{equation}\label{eq:zen U}
  Z=\{\:I_{n+1}+aE_{1,n+1}\:\mid\:a\in\F_p\:\}.  
 \end{equation}

For a homomorphism of pro-$p$ groups $\rho\colon G\to\dbU_{n+1}$, and for $1\leq i\leq n$, let $\rho_{i,i+1}$ denote the projection of $\rho$ on the $(i,i+1)$-entry.
Observe that $\rho_{i,i+1}\colon G\to\F_p$ is a homomorphism of pro-$p$ groups, and thus we may consider $\rho_{i,i+1}$ as an element of $\rmH^1(G,\F_p)$.
One has the following ``pro-$p$ translation'' of Dwyer's result on Massey products (cf., e.g., \cite[Lemma~9.3]{eq:kummer}, see also \cite[\S~8]{ido:Massey}).

\begin{prop}\label{prop:masse unip}
Let $G$ be a pro-$p$ group and let $\alpha_1,\ldots,\alpha_n$ be a sequence of elements of $\rmH^1(G,\F_p)$, with $n\geq2$.
\begin{itemize}
 \item[(i)] The $n$-fold Massey product $\langle\alpha_1,\ldots,\alpha_n\rangle$ is defined if, and only if, there exists a continuous homomorphism $ \bar\rho\colon G\to\bar\dbU_{n+1}$ such that $\bar\rho_{i,i+1}=\alpha_i$ for every $i=1,\ldots,n$.
 \item[(ii)] The $n$-fold Massey product $\langle\alpha_1,\ldots,\alpha_n\rangle$ vanishes if, and only if, there exists a continuous homomorphism $ \rho\colon G\to\dbU_{n+1}$ such that $\rho_{i,i+1}=\alpha_i$ for every $i=1,\ldots,n$.
\end{itemize}

\end{prop}
 
By Proposition~\ref{prop:Massey cup}--(i), in order to check that a pro-$p$ group satisfies the $n$-Massey vanishing property for some $n\geq2$, it suffices to verify that every defined $n$-fold Massey product associated to a sequence of {\sl non-trivial } cohomology elements of degree 1 vanishes. 
Analogously, we use Proposition~\ref{prop:masse unip} to show that, in order to check that a pro-$p$ group satisfies the strong $n$-Massey vanishing property, it suffices to verify that every sequence of length at most $n$ of {\sl non-trivial} cohomology elements of degree 1 whose cup-products satisfy the triviality condition \eqref{eq:cup 0} yields a Massey product containing 0.
 
 \begin{prop}\label{prop:trivial strong}
Given $N>2$, a pro-$p$ group $G$ satisfies the strong $n$-Massey vanishing property for every $3\leq n\leq N$ if, and only if,  for every $3\leq n\leq N$, every sequence $\alpha_1,\ldots,\alpha_n$ of non-trivial elements of $\rmH^1(G,\F_p)$ satisfying the triviality condition \eqref{eq:cup 0} yields an $n$-fold Massey product containing $0$. 
 \end{prop}

 \begin{proof}
  If $G$ satisfies the strong $n$-Massey vanishing property, then obviously $\langle\alpha_1,\ldots,\alpha_n\rangle$ vanishes for every sequence $\alpha_1,\ldots,\alpha_n$ with $\alpha_1,\ldots,\alpha_n\neq0$ satisfying the triviality condition \eqref{eq:cup 0}.
  
 So assume that for every sequence $\alpha_1,\ldots,\alpha_n$, with $\alpha_i\neq0$ for every $i=1,\ldots,n$ and satisfying the triviality condition \eqref{eq:cup 0}, one has $0\in\langle\alpha_1,\ldots,\alpha_n\rangle$.
 Now pick an arbitrary sequence $\alpha_1,\ldots,\alpha_m$ with $\alpha_i\in\rmH^1(G,\F_p)$ and $3\leq m\leq N$ such that $\alpha_i\smallsmile\alpha_{i+1}=0$ for every $i=1,\ldots,m-1$.
 Let $$0=m_1<m_2<\ldots<m_r<m_{r+1}=m$$
 be positive integers such that: either
 $$\alpha_1=\alpha_{m_1+1},\ldots,\alpha_{m_2}\neq0,\qquad \alpha_{m_2+1},\ldots,\alpha_{m_3}=0,\qquad \alpha_{m_3+1},\ldots,\alpha_{m_4}\neq0,\qquad\ldots$$
 and so on, if $\alpha_1\neq0$; or conversely
 $$\alpha_1=\alpha_{m_1+1},\ldots,\alpha_{m_2}=0,\qquad \alpha_{m_2+1},\ldots,\alpha_{m_3}\neq0,\qquad \alpha_{m_3+1},\ldots,\alpha_{m_4}=0,\qquad\ldots$$
 and so on, if $\alpha_1=0$.
For every $j=1,\ldots,r$ put $n_j=m_{j+1}-m_j$.
By hypothesis, for every $j$ such that $\alpha_{m_j+1},\ldots,\alpha_{m_{j+1}}\neq0$, the $n_j$-fold Massey product $\langle\alpha_{m_j+1},\ldots,\alpha_{m_{j+1}}\rangle$ vanishes.
 Hence, by Proposition~\ref{prop:masse unip}--(ii) there exists a homomorphism $\rho_j\colon G\to\dbU_{n_j+1}$ such that $(\rho_{j})_{i,i+1}=\alpha_{m_j+1}$ for every $i=1,\ldots,n_j$.
 On the other hand, if $j$ is such that $\alpha_{m_j+1},\ldots,\alpha_{m_{j+1}}=0$ and $n_j>1$, then we set $\rho_{j}\colon G\to\dbU_{n_j-1}$ to be the homomorphism constantly equal to $I_{n_j-1}$.
 
Thus, we may define blockwise a homomorphisms $\rho\colon G\to\dbU_{m+1}$, where
\[
 \rho=\left(\begin{array}{cccc} \rho_{1} &&&0 \\ &\rho_{2}&& \\ &&\ddots& \\ 0&&&\rho_{r}
            \end{array} \right),\]
where we omit $\rho_{j}$ if $\alpha_{m_j+1}=0$, $n_j=1$, and $j\neq 1,r$.
For example, if $\alpha_1\neq0$ and $n_2=1$ then one has
\[
 \rho={\tiny \left(\begin{array}{@{}c|cc@{}} 
 \rho_1\left\{\begin{array}{ccc}   \ddots & \ddots & \\ &1&\alpha_{m_2} \\ &&1      \end{array} \right. &
\begin{array}{ccccc}  &&&& \\ &&&& \\ \alpha_{m_2+1}&&&&      \end{array}  & \\ \hline
  & \left.\begin{array}{cc}   1&\alpha_{m_3+1} \\ &\ddots  \end{array}\right\}\rho_3 & \\ 
  & &\ddots \end{array} \right)},
\]
where $m_2+1=m_3$ and $\alpha_{m_2+1}=0$.
Then $$\rho_{i,i+1}=(\rho_{j})_{i-m_j,i-m_j+1}=\alpha_i,\qquad\text{if }m_j<i\leq m_{j+1},\:\alpha_{m_j+1}\neq0,$$
and $\rho_{i,i+1}=0=\alpha_i$ otherwise.
Therefore, Proposition~\ref{prop:masse unip}--(ii) implies that the $m$-fold Massey product $\langle\alpha_1,\ldots,\alpha_m\rangle$ vanishes.
 \end{proof}

 %%%%%%%%%%%%%%%%%%%%%%%%%%%%5
 
 \subsection{Commutators of upper unitriangular matrices}\label{ssec:unipot}
 Let $L=\bigoplus_{k\geq1}L_k$ be a graded Lie algebra over $\F_p$.
 Suppose that $L_k=0$ for $k>n$, for some positive $n$, and that for $1\leq k\leq n$ every subspace $L_k$ has dimension $n+1-k$.
 Suppose further that each non-trivial subspace $L_k$ has a basis 
 $$\left\{\:\eue_{1,1+k},\:\eue_{2,2+k},\:\ldots,\:\eue_{n+1-k,n+1}\:\right\}\subseteq L_k,$$ whose elements satisfy
 \begin{equation}\label{eq:liebrackets}
   [\eue_{i,j},\eue_{i',j'}]=\begin{cases}
                             \eue_{i,j'}, & \text{if }j=i',\\ -\eue_{i',j}, & \text{if }i=j',\\ 0,&\text{otherwise}
                            \end{cases} \end{equation}
--- here $[\textvisiblespace,\textvisiblespace]$ denotes the Lie brackets of $L$ --- for every $1\leq i,i'\leq n$ and $2\leq j,j'\leq n+1 $ such that $i<j$ and $i'<j'$.
 Then one has the following.
 
 \begin{lem}\label{lem:Liealg}
  Let $L=\bigoplus_{k\geq 1}L_k$ be a graded Lie $\F_p$-algebra as above, and let $a$ be an element of $L_1$ such that 
  $$a=\sum_{i=1}^na_i\eue_{i,i+1},\qquad a_1,\ldots,a_n\in \F_p$$
satisfying $a_1a_2\cdots a_{n+1-k'}\neq0$ or $a_{k'}a_{k'+1}\cdots a_n\neq0$ for some $k'$, $2\leq k'\leq n$.
  Then for every $c\in L_{k'}$ there exists $b\in L_{{k'}-1}$ such that $[b,a]=c$.
 \end{lem}
 
 \begin{proof}
  Write $c=\sum_{l=1}^{n+1-{k'}}c_l\eue_{l,l+k}$, with $c_l\in \F_p$, and
  \[
   b=\sum_{j=1}^{n+2-{k'}}b_j\eue_{j,j+{k'}-1},\qquad b_j\in \F_p
  \]
for an arbitrary element $b\in L_{{k'}-1}$.
Then applying \eqref{eq:liebrackets}  yields
\[\begin{split}
   [b,a]&=\sum_{j=1}^{n+2-{k'}}\sum_{i=1}^{n}b_ja_i[\eue_{j,j+{k'}-1},\eue_{i,i+1}]\\
   &=\sum_{j=1}^{n+1-{k'}}b_ja_{j+{k'}-1}\eue_{j,j+{k'}}-\sum_{j=2}^{n+2-{k'}}b_ja_{j-1}\eue_{j-1,j+{k'}-1}\\
   &=\sum_{j=1}^{n+1-{k'}}\left(b_ja_{j+{k'}-1}-b_{j+1}a_j\right)\eue_{j,j+{k'}}.
  \end{split}\]
Therefore, $[b,a]=c$ if, and only if, the system
\[
\left( \begin{array}{cccccc} a_{k'} & -a_1&0&\cdots& 0\\
 & a_{{k'}+1} & -a_2  &\cdots &0\\
&& \ddots & \ddots &\vdots \\
&&&a_n& -a_{n+1-{k'}}\end{array}\right)\cdot 
\left( \begin{array}{c} b_1 \\ b_2 \\ \vdots \\ b_{n+1-{k'}}\\ b_{n+2-{k'}}\end{array}\right)
 =\left( \begin{array}{c} c_1 \\ c_2 \\ \vdots \\ c_{n+1-{k'}}\end{array}\right)
\]
--- where $b_1,\ldots,b_{n+2-{k'}}$ are the indeterminates --- is solvable.
The condition on the coefficients $a_1,\ldots,a_n$ ensures that rank of the matrix of coefficients is $n+1-{k'}$; so the system is solvable, yielding a suitable $b\in L_{{k'}-1}$. 
 \end{proof}
 
 Given an arbitrary group $U$, and two elements $g,h\in U$, we write ${}^gh=ghg^{-1}$, and $[g,h]={}^gh\cdot h^{-1}=ghg^{-1}h^{-1}$.
 Given three elements $g_1,g_2,h$ of $U$, one has the identities
 \begin{equation}\label{eq:comm0}
 \begin{split}
    [g_1g_2,h] &= \left[g_1,[g_2,h]\right][g_2,h][g_1,h],\\ [h,g_1g_2]&=[h,g_1]\left[g_1,[h,g_2]\right][h,g_2].
 \end{split}
 \end{equation}
Let $(U_{(k)})_{k\geq1}$ denote the descending central series of $U$, i.e.,
$$U_{(1)}=U\qquad\text{and}\qquad U_{(k+1)}=[U_{(k)},U].$$

 Within this subsection, we fix an integer $n\geq 2$, and for simplicity we write $\dbU$ instead of $\dbU_{n+1}$, and $I$ instead of $I_{n+1}$.
 The following properties of $\dbU$ are well-known (cf., e.g., \cite[Thm.~1.5]{matrix1}, \cite[\S~2]{idomatrix}, and \cite[\S~4]{pal:Massey}).
For every $k\geq1$, the $k$-th term $\dbU_{(k)}$ of the descending central sequence of $\dbU$ is the subgroup
   \[
     \dbU_{(k)} = \left\{\:I+\sum_{j-i\geq k}a_{ij}E_{ij}\:\mid\:a_{ij}\in\F_p\:\right\}.
   \]
In particular, $\dbU_{(n)}$ is the subgroup $Z$ of $\dbU$ defined in \eqref{eq:zen U}, while $\dbU_{(k)}=0$ for $k>n$.
Moreover, every quotient $\dbU_{(k)}/\dbU_{(k+1)}$ is a $p$-elementary abelian group.
Altogether, the graded $\F_p$-vector space
\[
 L(\dbU)=\bigoplus_{k\geq1} L(\dbU)_{k},\qquad L(\dbU)_{k}=\dbU_{(k)}/\dbU_{(k+1)}
\]
is a graded Lie algebra over $\F_p$, endowed with the Lie brackets induced by commutators.
Moreover, for every $1\leq k\leq n$ one has $\dim(L(\dbU)_{k})=n+1-k$, and $L(\dbU)_{k}$ comes endowed with a basis 
$\{\eue_{1,1+k},\ldots,\eue_{n+1-k,n+1}\}$ with 
$$\eue_{i,i+k}=(I+E_{i,i+k})\dbU_{k+1}\in L(\dbU)_{k},\qquad \text{with }1\leq i\leq n+1-k.$$
 Straightforward computations show that the elements $\eue_{i,i+k}$ above satisfy \eqref{eq:liebrackets}
--- cf., e.g., \cite[Lemma~4.3]{pal:Massey}.
From this we deduce the following.

\begin{prop}\label{prop:commutators}
Let $A\in\dbU$ be the matrix with coset $a\in L(\dbU)_1$, $a=\sum_{i=1}^na_i\eue_{i,i+1}$, with $a_i\neq0$ for all $i=1,\ldots,n$, and let $k$ be a positive integer such that $3\leq k\leq n$.
For every $C\in\dbU_{(k)}$ there exists a matrix $B\in\dbU_{(k-1)}$ such that 
\begin{equation}\label{eq:equality comm}
 [B,A]=C.
\end{equation}
 \end{prop}
 
 \begin{proof}
 For $l\geq1$ we produce matrices $B_1,\ldots,B_l\in\dbU_{(k-1)}$ satisfying
\begin{equation}\label{eq:Bl mod Uk}
 [B_l\cdots B_2 B_1,A]\equiv C\mod \dbU_{(k+l)}.
\end{equation} 
Since $k+l\geq n+1$ for $l$ sufficiently large, one has $\dbU_{(k+l)}=\{1\}$, and thus from \eqref{eq:Bl mod Uk} one obtains $[B_l\cdots B_1,A]=C$.
So we may put $B=B_l\cdots B_1$, so that $B$ satisfies \eqref{eq:equality comm}.

Observe that the coset $a\in L(\dbU)_1$ of $A$ satisfies the hypothesis of Lemma~\ref{lem:Liealg}.
Let $c\in L(\dbU)_{k}$ be the coset of $C$.
Lemma~\ref{lem:Liealg} yields an element $b\in L(\dbU)_{k-1}$ such that $[b,a]=c$.
Therefore, any matrix $B_1\in\dbU_{(k-1)}$ with coset $b$ satisfies \eqref{eq:Bl mod Uk} with $l=1$.
Now suppose that $l\geq1$, and that we have found $l$ matrices $B_1,\ldots, B_l\in\dbU_{(k-1)}$ satisfying \eqref{eq:Bl mod Uk}. 
Namely, one has
$$C_l:=[B_{l}\cdots B_2 B_1,A]^{-1}C\in\dbU_{(k+l)}.$$
Then again Lemma~\ref{lem:Liealg} yields $B_{l+1}\in \dbU_{(k+l-1)}$ --- hence $B_{l+1}$ lies in $\dbU_{(k-1)}$, too --- such that 
$[B_{l+1},A]\equiv C_l\bmod \dbU_{(k+l+1)}$.
The commutator identities \eqref{eq:equality comm} imply 
\[\begin{split}
[B_{l+1}\cdot (B_{l}\cdots B_2 B_1),A]&=
[B_{l+1},[B_{l}\cdots B_2 B_1,A]]\cdot[B_{l}\cdots B_2 B_1,A]\cdot[B_{l+1},A]  \\
&\equiv [B_{l+1},[B_{l}\cdots B_2 B_1,A]]\cdot C\mod\dbU_{(k+l+1)}\\
&\equiv C\mod\dbU_{(k+l+1)},
\end{split}\]
as $[B_{l+1},[B_{l}\cdots B_2 B_1,A]]\in\dbU_{(k+l-1)+k}$ and $2k+l-1\geq k+l+1$.
Altogether, $B_1,\ldots,B_{l+1}$ lie in $\dbU_{(k-1)}$, and they satisfy \eqref{eq:Bl mod Uk} (with $l+1$ instead of $l$).
 \end{proof}

 \section{Oriented pro-$p$ groups}
 
 Recall that, given a pro-$p$ group $G$, the Frattini subgroup $\Phi(G)$ of $G$ is the subgroup $G^p\cdot\mathrm{cl}([G,G])$, where the latter factor is the closure of the commutator subgroup $[G,G]$ with respect to the topology of $G$.
 By \eqref{eq:H1}, one has an isomorphism of $\F_p$-vector spaces
\begin{equation}\label{eq:H1 dual}
 \rmH^1(G,\F_p)=(G/\Phi(G))^\ast,
\end{equation}
where $\textvisiblespace^\ast$ denotes the $\F_p$-dual space (c.f., e.g., \cite[Ch.~I, \S~4.2, p.~29]{serre:galc}).
 %%%%%%%%%%%%%%%%%%%5
 
 \subsection{Orientations}\label{ssec:or}
 Recall that $1+p\Z_p$ denotes the multiplicative group of principal units of the ring of $p$-adic integers $\Z_p$, i.e.,
 \[
  1+p\Z_p=\{\:1+p\lambda\:\mid\:\lambda\in\Z_p\:\}.
 \]
If $p=2$ then $1+2\Z_2=\{\pm1\}\times(1+4\Z_2)$, which is isomorphic to $(\Z/2)\times \Z_2$ as an abelian pro-2 group; while $1+p\Z_p$ is a free cyclic pro-$p$ group if $p\neq2$.

Let $G$ be a pro-$p$ group.
A homomorphism of pro-$p$ groups $\theta\colon G\to1+p\Z_p$ is called an {\sl orientation}, and the pair $(G,\theta)$ is called an {\sl oriented pro-$p$ group} (cf. \cite{qw:cyc}; oriented pro-$p$ groups were introduced by I.~Efrat in \cite{efrat:small}, with the name ``cyclotomic pro-$p$ pairs'').
An orientation $\theta\colon G\to1+p\Z_p$ of a pro-$p$ group $G$ is said to be {\sl torsion-free} if $p\neq2$, or if $p=2$ and $\Img(\theta)\subseteq1+4\Z_2$.

If $(G,\theta)$ and $(H,\tau)$ are two oriented pro-$p$ groups, a homomorphism of oriented pro-$p$ groups 
$$\phi\colon (G,\theta)\longrightarrow(H,\tau)$$ is a homomorphism of pro-$p$ groups $\phi\colon G\to H$ such that $\theta=\tau\circ\phi$.
 
 \begin{exam}\label{exam:cyclo char}\rm
The maximal pro-$p$ Galois group $G_{\K}(p)$ of a field $\K$ containing a root of 1 of order $p$ comes endowed naturally with an orientation: namely, the {\sl $p$-cyclotomic character} $\theta_{\K}\colon G_{\K}(p)\to1+p\Z_p$, satisfying 
\[
 g(\zeta)=\zeta^{\theta_{\K}(g)}\qquad\text{for every }g\in G_{\K}(p),
\]
for any root $\zeta\in\K(p)$ of 1 of order a power of $p$ (cf. \cite[\S~4]{eq:kummer}).
The image of $\theta_{\K}$ is $1+p^f\Z_p$, where $f$ is the maximal positive integer such that $\K$ contains the roots of 1 of order $p^f$ --- if such a number does not exists, i.e., if $\K$ contains all roots of 1 of $p$-power order, then $\Img(\theta_{\K})=\{1\}$, and one sets $f=\infty$.
Observe that if $p\neq2$, or if $p=2$ and $\sqrt{-1}\in\K$, then $\theta_{\K}$ is a torsion-free orientation.
\end{exam}

 From now on, given an orientation $\theta\colon G\to1+p\Z_p$ of a pro-$p$ group $G$ the notation $\Img(\theta)=1+p^\infty\Z_p$ will mean that the image of $\theta$ is trivial. 
 
 \begin{exam}\label{exam:Demushkin}\rm
 A {\sl Demushkin group} is a pro-$p$ group $G$ satisfying the following:
\begin{itemize}
 \item[(i)] $\dim(\rmH^1(G,\F_p))<\infty$;
 \item[(ii)] $\rmH^2(G,\F_p)\simeq\F_p$;
 \item[(iii)] the cup-product induces a non-degenerate bilinear form
 \[ \xymatrix{\rmH^1(G,\F_p)\times \rmH^1(G,\F_p)\ar[r] &  \rmH^2(G,\F_p)};  \]
\end{itemize}
cf., e.g., \cite[Def.~3.9.9]{nsw:cohn}.
J-P.~Serre proved that every Demushkin group comes endowed with a canonical orientation $\theta_G\colon G\to1+p\Z_p$ which completes $G$ into an oriented pro-$p$ group (cf. \cite{serre:Demushkin}).
If the canonical orientation $\theta_G$ is torsion-free, then 
\begin{equation}\label{eq:pres Demushkin}
 G=\left\langle\: x_1,\ldots,x_d\:\mid\:x_1^{p^f}[x_1,x_2]\cdots[x_{d-1},x_d]=1\:\right\rangle
\end{equation}
for some even positive integer $d$, and with $f\in\dbN\cup\{\infty\}$ such that $\Img(\theta_G)=1+p^f\Z_p$ (cf., e.g., \cite[Thm.~3.9.11]{nsw:cohn}), and $\theta_G(x_2)=1+p^f$ and $\theta_G(x_h)=1$ for $h\neq2$ (see also \cite[\S~5.3]{qw:cyc}).
%On the other hand, if $\Img(\theta_G)=\{\pm1\}$ then either
%\begin{equation}\label{eq:pres Demushkin pm1 a}
%G=\left\langle\: x_1,\ldots,x_d\:\mid\:x_1^{2}[x_2,x_3]\cdots[x_{d-1},x_d]=1\:\right\rangle
%\end{equation}
%with $d$ odd, and $\theta_G(x_1)=-1$ and $\theta_G(x_h)=1$ for $h\neq1$ (in particular, if $d=1$ then $G\simeq\Z/2$); or
%\begin{equation}\label{eq:pres Demushkin pm1 b}
% G=\left\langle\: x_1,\ldots,x_d\:\mid\:x_1^{2}[x_1,x_2]\cdots[x_{d-1},x_d]=1\:\right\rangle
%\end{equation}
%with $d$ even, and $\theta_G(x_2)=-1$ and $\theta_G(x_h)=1$ for $h\neq2$ (cf. \cite[Thm.~3--4]{labute}).
\end{exam}

\begin{rem}\label{rem:Demushkin}\rm
If $\K$ is an $\ell$-adic local field, with $\ell$ a prime different to $p$ --- respectively if $\K$ is a $p$-adic local field ---, containing a root of 1 of order $p$, then its maximal pro-$p$ Galois group is a Demushkin group $G$, with 
$$\dim(\rmH^1(G,\F_p))=\begin{cases} 2, \\ [\K:\Q_p]+2 \end{cases}
$$
respectively --- in particular, in the former case (i.e., $\K$ is $\ell$-adic) one has $G\simeq\Z_p\rtimes\Z_p$ --- (cf. \cite[Prop.~7.5.9, Thm.~7.5.11]{nsw:cohn}).
In this case the canonical orientation $\theta_G$ coincides with the pro-$p$ cyclotomic character $\theta_{\K}$ (see Example~\ref{ex:kummer}--(b) below).
Also, $\Z/2$ is the maximal pro-$2$ Galois group of $\dbR$.
It is still an open problem to determine whether {\sl any other} Demushkin group occurs as the maximal pro-$p$ Galois group of a field containing a root of 1 of order $p$: for example, the simplest example for which this is not known is the Demushkin pro-2 group
\[
 G=\langle\:x_1,x_2,x_3\:\mid\:x_1^2[x_2,x_3]=1\:\rangle
\]
(cf. \cite[Rem.~5.5]{jacobware}); while the only Demushkin group on 4 generators which is known to be realizable as a
maximal pro-$p$ Galois group is the pro-3 group
\[
  G=\langle\:x_1,x_2,x_3,x_4\:\mid\:x_1^3[x_1,x_2][x_3,x_4]=1\:\rangle,
\]
which occurs as the maximal pro-$3$ Galois group of $\Q_3(\zeta_3)$, where $\zeta_3$ is a root of 1 of order 3 (cf. \cite[p.~254]{Koenigsmann}).
\end{rem}

%%%%%%%%%%%%%%%%%%%%%%%%%%%%%%%%%%%%%%%%%%%%%%%
\subsection{Oriented pro-$p$ groups of elementary type}\label{ssec:etc}

In the family of oriented pro-$p$ groups one has the following two constructions (cf. \cite[\S~3]{efrat:small}).
\begin{itemize}
 \item[(a)] Let $(G_0,\theta)$ be an oriented pro-$p$ group, and let $A$ be a free abelian pro-$p$ group.
 The {\sl semidirect product} $(A\rtimes_\theta G_0,\tilde\theta)$ is the oriented pro-$p$ group where $A\rtimes_\theta G_0$ is the semidirect product of pro-$p$ groups with action $gag^{-1}=a^{\theta(g)}$ for every $g\in G_0$ and $a\in A$, and where 
 $$\tilde\theta\colon A\rtimes_\theta G_0\longrightarrow 1+p\Z_p$$
 is the orientation induced by $\theta$, i.e., $\tilde\theta=\theta\circ\pi$, where $\pi\colon A\rtimes_\theta G_0\to G_0$ is the canonical projection.
 \item[(b)] Let $(G_1,\theta_1),(G_2,\theta_2)$ be two oriented pro-$p$ groups.
 The {\sl free product} $(G_1\ast G_2,\theta)$ is the oriented pro-$p$ group where $G_1\ast G_2$ denote the free pro-$p$ product of the two pro-$p$ groups $G_1,G_2$, while $$\theta\colon G_1\ast G_2\longrightarrow1+p\Z_p$$
 is the orientation induced by the orientations $\theta_1,\theta_2$ via the universal property of the free pro-$p$ product.
\end{itemize}

\begin{defin}\label{defin:ET}\rm
  The family of {\sl oriented pro-$p$ groups of elementary type} is the smallest family of oriented pro-$p$ groups containing
 \begin{itemize}
  \item[(a)] every oriented pro-$p$ group $(F,\theta)$, where $F$ is a finitely generated free pro-$p$ group, and $\theta\colon F\to1+p\Z_p$ is arbitrary,
  \item[(b)] every Demushkin group endowed with its canonical orientation $(G,\theta_G)$ (cf. Example~\ref{exam:Demushkin});
 \end{itemize}
and such that 
\begin{itemize}
 \item[(c)] if $(G_0,\theta)$ is an oriented pro-$p$ group of elementary type, then also the semidirect product $(\Z_p\rtimes_\theta G_0,\tilde\theta)$ is an oriented pro-$p$ group of elementary type,
 \item[(d)]if $(G_1,\theta_1),(G_2,\theta_2)$ are two oriented pro-$p$ groups of elementary type, then also the free product $(G_1\ast G_2,\theta)$ is an oriented pro-$p$ group of elementary type.
\end{itemize}
\end{defin}

\begin{rem}\label{rem:subgroups et}\rm
\begin{itemize}
 \item[(a)] If $(G,\theta)$ is an oriented pro-$p$ group of elementary type, and $H$ is a finitely generated subgroup of $G$, then also the oriented pro-$p$ group $(H,\theta\vert_H)$ is of elementary type (cf., e.g., \cite[Rem.~5.10--(b)]{qw:cyc}).
\item[(b)]  Given an oriented pro-$p$ group of elementary type $(G,\theta)$, there might be another orientation $\tau\colon G\to1+p\Z_p$, $\tau\neq\theta$, such that also $(G,\tau)$ is of elementary type --- e.g., if $G=F$ is a finitely generated free pro-$p$ group.
%It is not difficult to see that this happens if, and only if, the underlying pro-$p$ group $G$
%decomposes as a free pro-$p$ product with at least a free factor --- namely, there exist a free subgroup $F\subseteq G$, $F\neq\{1\}$, and another subgroup $H\subseteq G$ such that $G=F\ast H$.
\end{itemize}
\end{rem}

I.~Efrat's Elementary Type Conjecture asks whether the maximal pro-$p$ Galois group $G_{\K}(p)$ of every field $\K$ containing a root of 1 of order $p$ such that $[\K^\times:(\K^\times)^p]<\infty$ may be obtained in this way.
More precisely, the conjecture states the following (cf. \cite{ido:etc}, see also \cite[Question~4.8]{ido:etc2}, \cite[\S~10]{marshall}, \cite[\S~7.5]{qw:cyc} and \cite[Conj.~4.8]{MPQT}).

\begin{conj}\label{conj:etc}
In Definition~\ref{defin:ET}, replace item~{\rm(b)} with 
\begin{itemize}
 \item[(b')] every oriented pro-$p$ group $(G_{\K}(p),\theta_{\K})$, where $\K$ is a $p$-adic field containing a root of 1 (cf. Remark~\ref{rem:Demushkin}), and also the oriented pro-2 group $(\Z/2,\theta_{\Z/2})$, with $\Img(\theta_{\Z/2})=\{\pm1\}$, if $p=2$.
\end{itemize}
Then the family of oriented pro-$p$ groups $(G_{\K}(p),\theta_{\K})$, where $\K$ is a field containing a root of 1 of order $p$ such that $[\K^\times:(\K^\times)^p]<\infty$, coincides with the family of oriented pro-$p$ groups obtained from Definition~\ref{defin:ET}, with {\rm(b')} instead of item~{\rm(b)}.
\end{conj}

On the one hand, one knows that all oriented pro-$p$ groups constructed as in Conjecture~\ref{conj:etc} occur as maximal pro-$p$ Galois groups (endowed with the pro-$p$ cyclotomic character), as the realizability as maximal pro-$p$ Galois group is preserved by semidirect products and free pro-$p$ products (cf. \cite[Rem.~3.4]{efrat:small}).
On the other hand, one has the following.

\begin{prop}\label{prop:etc fields}
 Let $\K$ be a field satisfying $[\K^\times:(\K^\times)^p]<\infty$ and containing a root of 1 of order $p$.
 %{\rm(}and also $\sqrt{-1}$ if $p=2${\rm)}
 Then the oriented pro-$p$ group $(G_{\K}(p),\theta_{\K})$ is of elementary type in the following cases:
\begin{itemize}
  \item[(i)] $\K$ is finite;
\item[(ii)] $\K$ is a PAC field, or an extension of relative transcendence degree 1 of a PAC field;
\item[(iii)] $\K$ is a local field, or an extension of transcendence degree 1 of a local field, with characteristic not $p$;
\item[(iv)] $\K$ is $p$-rigid {\rm (}cf. \cite[p.~722]{ware}{\rm)};
\item[(v)] $\K$ is an algebraic extension of a global field of characteristic not $p$;
\item[(vi)] $\K$ is a valued $p$-Henselian field with residue field $\kappa$, where $(G_{\kappa}(p),\theta_{\kappa})$ is of elementary type.
\item[(vii)] $\K$ is a Pythagorean field, if $p=2$.
 \end{itemize}
\end{prop}

For Proposition~\ref{prop:etc fields} see \cite[Thm.~D, Prop.~6.2--6.3]{MPQT}.
See also: Remark~\ref{rem:Demushkin} and \cite{efrat:funfield} for item~(iii); \cite[\S~3]{cmq:fast} for item~(iv); \cite{ido:etc2} for item~(v); \cite[\S~1]{EK} for item~(vi) in case $\mathrm{char}(\kappa)\neq p$, and \cite[\S~3]{henselian} for item~(vi) in case $\mathrm{char}(\kappa)= p$; and \cite[Thm.~6.5]{MPQT} for item~(vii).

 %%%%%%%%%%%%%%%%%%%%%%%%%%%%%%%%%%%
 
 \subsection{Kummerian oriented pro-$p$ groups}\label{ssec:kummer}
 
 An oriented pro-$p$ group $(G,\theta)$, with $\theta$ a torsion-free orientation, is said to be {\sl $\theta$-abelian} if $\Ker(\theta)$ is a free abelian pro-$p$ group, and there exists a complement $G_0\subseteq G$ to $\Ker(\theta)$ --- thus, $G_0\simeq\Img(\theta)$ ---, and 
 \begin{equation}\label{eq:def thetabelian}
  (G,\theta)\simeq \Ker(\theta)\rtimes_{\theta}(G_0,\theta\vert_{G_0})  
 \end{equation}
 (cf. \cite[\S~1]{cq:bk}).
Equivalently, $(G,\theta)$ is $\theta$-abelian if, and only if, $G$ has a presentation
\begin{equation}\label{eq:thetabelian pres}
  G=\left\langle\:x_0,x_h\:\mid\:h\in J,\:{}^{x_0}x_h=x_h^{\theta(x_0)},\:
 [x_h,x_l]=1\:\forall h,l\in J\:\right\rangle,
\end{equation}
for some set $J$ (cf. \cite[Prop.~3.4]{cq:bk}).
 
 The following notion was introduced in \cite{eq:kummer} (here we use the formulation given in \cite[\S~2]{qw:bogo}, which is the most useful for our purposes).
 
 \begin{defin}\label{defin:kummer}\rm
  An oriented pro-$p$ group $(G,\theta)$, with torsion-free orientation $\theta$, is said to be {\sl Kummerian} if there exists 
  an epimorphism of oriented pro-$p$ groups $$\phi\colon (G,\theta)\longrightarrow(\bar G,\bar\theta)$$ with $(\bar G,\bar\theta)$ a $\bar\theta$-abelian oriented pro-$p$ group, such that $\Ker(\phi)\subseteq\Phi(G)$.
  \end{defin}

% Roughly speaking, an oriented pro-$p$ group with torsion-free orientation is Kummerian if, and only if, it has an epimorphic image $(\bar G,\bar\theta)$ which is $\bar\theta$-abelian, and with ``the same'' generators.
By \eqref{eq:H1 dual}, $\Ker(\phi)\subseteq \Phi(G)$ if, and only if, $\phi$ yields an isomorphism
  \[
   \phi^\ast\colon \rmH^1(\bar G,\F_p)\overset{\sim}{\longrightarrow} \rmH^1(G,\F_p).
  \]

  Now let $\K$ be a field containing a root of 1 of order $p$ (and also $\sqrt{-1}$ if $p=2$), and let 
   $\sqrt[p^\infty]{\K}$ be the compositum of all extensions $\K(\sqrt[p^n]{a})$ with $a\in\K^\times$ and $n\geq1$.
  Then the restriction
  \[
   \phi\colon G_{\K}(p)=\Gal(\K(p)/\K)\longrightarrow\Gal(\sqrt[p^\infty]{\K}/\K)
  \]
induces an isomorphism $\phi^\ast\colon \rmH^1(\Gal(\sqrt[p^\infty]{\K}/\K),\F_p)\overset{\sim}{\to} \rmH^1(G,\F_p)$, as
$\Ker(\phi)\subseteq\Phi(G_{\K}(p))$;
moreover, by Kummer theory both cohomology groups are isomorphic to the quotient $\K^\times/(\K^\times)^p$.

Moreover, let $\K(\zeta_{p^\infty})$ be the compositum of all $p$-power cyclotomic extensions of $\K$.
Then either $\Gal(\K(\zeta_{p^\infty})/\K)$ is isomorphic to $\Z_p$, or it is trivial (if $\K$ contains all roots of 1 of $p$-power order).
Furthermore, again by Kummer theory one has
\[
 \Gal(\sqrt[p^\infty]{\K}/\K)\simeq\Gal(\sqrt[p^\infty]{\K}/\K(\zeta_{p^\infty}))\rtimes\Gal(\K(\zeta_{p^\infty})/\K),
\]
where $\Gal(\sqrt[p^\infty]{\K}/\K(\zeta_{p^\infty}))$ is a free abelian pro-$p$ group, and the action of the right-hand side factor on the left-hand side factor is induced by the $p$-cyclotomic character $\theta_{\K}$ --- i.e., 
\[
 ghg^{-1}=h^\lambda\qquad\forall\:g\in\Gal(\K(\zeta_{p^\infty})/\K),\:h\in\Gal(\sqrt[p^\infty]{\K}/\K(\zeta_{p^\infty})),
\]
where $\lambda\in1+p\Z_p$ is defined by $g(\zeta)=\zeta^\lambda$ with $\zeta\in\K(\zeta_{p^\infty})$ a root of 1 of order a power of $p$. 
Therefore, the oriented pro-$p$ group $(\Gal(\sqrt[p^\infty]{\K}/\K),\bar\theta)$ is $\bar\theta$-abelian, where $\bar\theta$ is the orientation satisfying $\theta_{\K}=\bar\theta\circ\phi$.
Thus, the oriented pro-$p$ group $(G_{\mathbb{K}}(p),\theta_{\mathbb{K}})$ is Kummerian (cf. \cite[Thm.~4.2]{eq:kummer} and \cite[Thm.~2.8]{qw:bogo}).
Moreover, for every $p$-extension $\mathbb{L}/\K$, $\K$ can be replaced by $\mathbb{L}$ and thus also the oriented pro-$p$ group $(G_{\mathbb{L}}(p),\theta_{\mathbb{L}})$ is Kummerian --- we underline that $\theta_{\mathbb{L}}$ is the restriction of $\theta_{\K}$ to $G_{\mathbb{L}}(p)$.

  One has also the following examples of Kummerian oriented pro-$p$ groups.
  
  \begin{exam}\label{ex:kummer}\rm
  \begin{itemize}
  \item[(a)] If $G$ is a finitely generated free pro-$p$ group, then the oriented pro-$p$ group $(G,\theta)$ is Kummerian for any torsion-free orientation $\theta\colon G\to1+p\Z_p$ (cf. \cite[Prop.~5.5]{eq:kummer}).
  Indeed, the subgroup
\[
 N=\left\{\:[g,h]h^{\theta(g)^{-1}-1}\:\mid\:g\in G,\:h\in\Ker(\theta)\:\right\}\subseteq G
\]
is a normal subgroup of $G$ contained in both $\Phi(G)$ and $\Ker(\theta)$, and the quotient $G/N$ has a presentation as in \eqref{eq:thetabelian pres}.
 
 \item[(b)] If $G$ is a Demushkin pro-$p$ group whose canonical orientation $\theta_G$ is torsion-free, then $\theta_G$ is the only orientation which completes $G$ into a Kummerian oriented pro-$p$ group (cf. \cite[Thm.~4]{labute}, see also \cite[Thm.~7.6]{eq:kummer}).
 In particular, if $G$ occurs as the maximal pro-$p$ Galois group of a field containing a root of 1 of order $p$, then $\theta_G=\theta_{\K}$.
  \item[(c)] If $(G_0,\theta)$ is a Kummerian oriented pro-$p$ group, with $\theta$ a torsion-free orientation, and $A\simeq\Z_p$, then also $(A\rtimes_\theta G_0,\tilde\theta)$ is Kummerian (cf. \cite[Prop.~3.6]{eq:kummer}).
  Also, if $(G_1,\theta_1)$ and $(G_2,\theta_2)$ are Kummerian oriented pro-$p$ groups, with $\theta_1,\theta_2$ torsion-free orientations, then also $(G_1\ast G_2,\theta)$ is Kummerian (cf. \cite[Prop.~7.5]{eq:kummer}).
\item[(d)] By the previous examples, every oriented pro-$p$ group of elementary type with torsion-free orientation $\theta$ is Kummerian (cf. \cite[\S~7]{eq:kummer} and \cite[\S~5.3]{qw:bogo}).
  \end{itemize}
\end{exam}

  Let $\K$ be a field containing a root of 1 of order $p$ (and also $\sqrt{-1}$ if $p=2$).
Then it is well-known that for every $\alpha\in\rmH^1(G_{\K}(p),\F_p)$ one has $\alpha\smallsmile\alpha=0$ (cf. Remark~\ref{rem:cup gradcomm}, and, e.g., \cite[Rem.~4.2]{qw:bogo} if $p=2$).
In \cite[Thm.~8.1]{mt:conj}, J.~Mina\v{c} and N.D.~T\^an proved the following: for every $\alpha\in\rmH^1(G_{\K}(p),\F_p)$ and for every $n>2$, the $n$-fold Massey product $\langle\alpha,\ldots,\alpha\rangle$ is defined, and vanishes.

We prove that pro-$p$ groups which may be completed into a Kummerian oriented pro-$p$ group with torsion-free orientation enjoy the same property.

\begin{thm}\label{thm:kummer Massey cyc}
 Let $(G,\theta)$ be a Kummerian oriented pro-$p$ group, with torsion-free orientation $\theta$.
 Then for every $\alpha\in\rmH^1(G_{\K}(p),\F_p)$ and for every $n>2$, the $n$-fold Massey product $\langle\alpha,\ldots,\alpha\rangle$ vanishes.
\end{thm}

\begin{proof}
First of all, observe that, if $p=2$, then $\alpha\smallsmile\alpha=0$ for every $\alpha\in\rmH^1(G,\F_2)$, as $(G,\theta)$ is Kummerian and $\theta$ is torsion-free (cf., e.g., \cite[Fact.~7.1]{qw:cyc}), while if $p\neq 2$ this is true anyway (cf. Remark~\ref{rem:cup gradcomm}), so that the sequence $\alpha,\ldots,\alpha$ of length $n$ satisfies the triviality condition \eqref{eq:cup 0}.

%%%%%%%%%%%%%%%%%5

Suppose first that $(G,\theta)$ is $\theta$-abelian.
Then $G$ has a presentation
\begin{equation}\label{eq:thetabelian pres proof1}
  G=\left\langle\:x_h\:\mid\:h\in J,\: [x_h,x_l]=1\:\forall h,l\in J\:\right\rangle
\end{equation}
for some set $J$, if $\Ker(\theta)=G$; or
\begin{equation}\label{eq:thetabelian pres proof2}
  G=\left\langle\:x_0,x_h\:\mid\:h\in J,\:[x_0,x_h]=x_h^{q},\:  [x_h,x_l]=1\:\forall h,l\in J\:\right\rangle,
\end{equation}
for some set $J$ and $q=p^f$ with $f\geq1$ (and $f\geq2$ if $p=2$), if $\Ker(\theta)\neq G$ (cf. \eqref{eq:thetabelian pres}).

Put $\dbU=\dbU_{n+1}$ and $I=I_{n+1}$, and set
\[
  A=I+\sum_{i=1}^nE_{i,i+1}={\small\left(\begin{array}{ccccc} 1 & 1 & & & 0\\ & 1 &1 && \\
  &&\ddots&\ddots& \\ &&&1&1 \\ &&&&1                      \end{array}\right)}\in\dbU.
\]
For every integers $a,b$ such that $0\leq a,b\leq p-1$, one has 
\begin{equation}\label{eq:equiv mod p A}
 A^a\equiv I+\sum_{i=1}^n\bar a\cdot E_{i,i+1}\mod \dbU_{(2)},
\end{equation}
where $\bar a\in\F_p$ denotes the class represented by $a$, and obviously $[A^a,A^b]=I$. 
Moreover, one has
\[ A^q=\begin{cases} I, &\text{if }q>n,\\
 I+\sum_{i=1}^{n+1-q}E_{i,i+q}\in\dbU_{(q)},      &\text{if }q\leq n,
     \end{cases}\]
where $q=p^f$ is as in \eqref{eq:thetabelian pres proof2}.
In the latter case, by Proposition~\ref{prop:commutators} there exists $B\in\dbU_{(q-1)}$ such that $[B,A]=A^q$.
 
Now for every $h\in J$ set $a_0,a_h\in\Z$ such that $0\leq a_j\leq p-1$ and $\bar a_j=\alpha(x_j)$ for every $j\in J\cup\{0\}$.
Observe that 
\[\begin{split}
   \left[BA^{a_0},A^{a_h}\right]&={}^B\left[A^{a_0},A^{a_h}\right]\cdot\left[B,A^{a_h}\right]\\
& =I\cdot \left({}^{A^{a_h-1}}[B,A]\cdots{}^A[B,A]\cdots[B,A]\right)\\
 &=A^{a_hq}.
  \end{split}\]
Then the assignment $\rho(x_h)=A^{a_h}$ for all $h\in J$, and $\rho(x_0)=BA^{a_0}$ if $\Ker(\theta)\neq G$, gives a continuous homomorphism $\rho\colon G\to\dbU$, which satisfies $\rho_{i,i+1}(x_j)=\bar a_j=\alpha(x_j)$ for every $i=1,\ldots,n$ and every $j\in J\cup\{0\}$ by \eqref{eq:equiv mod p A} ---
observe that $$\rho(x_0)\equiv A^{a_0}\mod \dbU_{(2)}$$ as $B\in\dbU_{(2)}$.
Hence the $n$-fold Massey product $\langle\alpha,\ldots,\alpha\rangle$ vanishes by Proposition~\ref{prop:masse unip}--(ii).

If $(G,\theta)$ is an arbitrary Kummerian oriented pro-$p$ group, then there exists an epimorphism of oriented pro-$p$ groups $\phi\colon(G,\theta)\to(\bar G,\bar \theta)$ with $(\bar G,\bar \theta)$ a $\bar\theta$-abelian pro-$p$ group and $\Ker(\phi)\subseteq\Phi(G)$.
Therefore, the inflation map 
$$\mathrm{inf}_{\bar G,G}^1\colon \rmH^1(\bar G,\F_p)\longrightarrow \rmH^1(G,\F_p)$$ 
induced by $\phi$ is an isomorphism, and there exists $\bar\alpha\in\rmH^1(\bar G,\F_p)$ such that $\mathrm{inf}_{\bar G,G}^1(\bar\alpha)=\alpha$, i.e. $\alpha=\bar\alpha\circ\phi$.
The argument above yields a continuous homomorphism $\rho\colon \bar G\to\dbU$ satisfying 
$\rho_{i,i+1}=\bar\alpha$ for every $i=1,\ldots,n$, and thus the continuous homomorphism $\rho\circ\phi\colon G\to\dbU$
satisfies 
$$(\rho\circ\phi)_{i,i+1}=\bar\alpha\circ\phi=\alpha\qquad\text{for every }i=1,\ldots,n.$$
Hence the $n$-fold Massey product $\langle\alpha,\ldots,\alpha\rangle$ vanishes by Proposition~\ref{prop:masse unip}--(ii).
\end{proof}

By Example~\ref{ex:kummer}--(d), Theorem~\ref{thm:kummer Massey cyc} implies the following.

\begin{cor}
Let $(G,\theta)$ be an oriented pro-$p$ group of elementary type with torsion-free orientation $\theta$, and let $\alpha$ be an element of $\rmH^1(G,\F_p)$.
Then for every $n\geq2$ the $n$-fold Massey product $\langle\alpha,\ldots,\alpha\rangle$ vanishes.
\end{cor}

 \section{Oriented pro-$p$ groups and vanishing of Massey products}
 %%%%%%%%%%%%%%%%%%%
 
 \subsection{Semidirect products}
 
 Let $(G_0,\theta_0)$ be an oriented pro-$p$ group, let $Z\simeq\Z_p$ be a cyclic pro-$p$ group, and set
 \begin{equation}\label{eq:wad}
  (G,\theta)=Z\rtimes_{\theta_0}(G_0,\theta_0).  
 \end{equation}
Then the first and second $\F_p$-cohomology groups of $G$ decompose as follow:
 \begin{eqnarray}\label{eq:semidirectprod cohom 1}
   \rmH^1(G,\F_p) &=& \rmH^1(G_0,\F_p)\oplus \rmH^1(Z,\F_p), \\
   \rmH^2(G,\F_p) &=& \rmH^2(G_0,\F_p)\oplus\left(\rmH^1(G_0,\F_p)\smallsmile \rmH^1(Z,\F_p)\right),\label{eq:semidirectprod cohom 2}
 \end{eqnarray}
(cf. \cite[Thm.~3.13]{qw:cyc}) --- observe that $\rmH^n(Z,\F_p)=0$ for every $n\geq2$.
Moreover, if $\{\chi_h\mid h\in J\}$ is a basis of $\rmH^1(G_0,\F_p)$, and $\psi$ generates $\rmH^1(Z,\F_p)\simeq\F_p$, then 
$$\left\{\:\chi_h\smallsmile\psi\:\mid\: h\in J\:\right\}$$
is a basis for $\rmH^1(G_0,\F_p)\smallsmile \rmH^1(Z,\F_p)$.

\begin{rem}\rm
 The description of the $\F_p$-cohomology of the semidirect product \eqref{eq:wad} has been provided first by A.~Wadsworth (cf.  \cite[Cor.~3.4 and Thm.~3.6]{wadsworth}). 
\end{rem}

 \begin{thm}\label{thm:semidirectprod}
Let $(G_0,\theta_0)$ be a Kummerian oriented pro-$p$ group with torsion-free orientation $\theta_0$, and let $Z\simeq\Z_p$ be a cyclic pro-$p$ group.
Set $G=Z\rtimes_{\theta_0}G_0$. 
\begin{itemize}
 \item[(i)] If $G_0$ satisfies the $n$-Massey vanishing property for every $n>2$, then also $G$ satisfies the $n$-Massey vanishing property for every $n>2$.
 \item[(ii)] If $G_0$ satisfies the strong $n$-Massey vanishing property for every $n>2$, then also $G$ satisfies the strong $n$-Massey vanishing property for every $n>2$.
\end{itemize}
 \end{thm}
 
 \begin{proof}
 First of all, since $(G_0,\theta_0)$ is Kummerian, also the semidirect product $(G,\theta)=Z\rtimes_{\theta_0} (G_0,\theta_0)$ is Kummerian, cf. Example~\ref{exam:Demushkin}--(c).
 Let $\pi\colon G\to G_0$ denote the canonical projection.
 
Let $\psi$ be a generator of $\rmH^1(Z,\F_p)$, and let $\alpha_1,\ldots\alpha_n$ be a sequence of non-trivial elements of $\rmH^1(G,\F_p)$
satisfying \eqref{eq:cup 0}.
  By \eqref{eq:semidirectprod cohom 1}, for every $i=1,\ldots,n$ one has $\alpha_i=\alpha_i\vert_{G_0}+b_i\psi$ for some
  $b_i\in\F_p$.
  Hence
  \begin{equation}\label{eq:cup trivial}
 0=\alpha_i\smallsmile\alpha_{i+1}=
 \underbrace{\left(\alpha_i\vert_{G_0}\smallsmile\alpha_{i+1}\vert_{G_0}\right)}_{\in\rmH^2(G_0,\F_p)}+
 \underbrace{\left(b_{i+1}\alpha_i\vert_{G_0}-b_i\alpha_{i+1}\vert_{G_0}\right)\smallsmile\psi}_{\in\rmH^1(G_0,\F_p)\smallsmile\psi}
  \end{equation}
for every $i=1,\ldots,n-1$.
By \eqref{eq:semidirectprod cohom 2}, equality \eqref{eq:cup trivial} holds if, and only if, 
$$\alpha_i\vert_{G_0}\smallsmile\alpha_{i+1}\vert_{G_0}=0\qquad \text{and}\qquad 
b_{i+1}\alpha_i\vert_{G_0}=b_i\alpha_{i+1}\vert_{G_0}$$
for every $i=1,\ldots,n-1$ --- indeed, for any $\alpha\in\rmH^1(G_0,\F_p)$, $\alpha\smallsmile\psi=0$ implies $\alpha=0$. 
Altogether, one has two cases:
\begin{itemize}
 \item[(a)] either $b_i=0$ and $$\alpha_i=\alpha_i\vert_{G_0}\circ\pi\neq0$$ for every $i=1,\ldots,n$;
 \item[(b)] or 
$$b_i\neq 0 \qquad\text{and}\qquad\alpha_i=\frac{b_i}{b_1}\left(\alpha_1\vert_{G_0}+b_1\psi\right)=\frac{b_i}{b_1}\cdot\alpha_1$$
for every $i=1,\ldots,n$
(recall that we are assuming that $\alpha_i\neq0$ for every $i$, so that if $b_i=0$ for some $i$ then $\alpha_i\vert_{G_0}\neq0$, and conversely if $\alpha_i\vert_{G_0}=0$ then $b_i\neq0$).
\end{itemize}
\medskip

\noindent {\sl Case (a).} 
Assume that the $n$-fold Massey product $\langle\alpha_1,\ldots,\alpha_n\rangle$ is defined in $\bfH^\bullet(G)$, and that $\alpha_i\neq0$ for every $i$ (cf. Proposition~\ref{prop:Massey cup}--(i)), to prove statement~(i).
By Proposition~\ref{prop:masse unip}--(i), there exists a homomorphism $\bar\rho\colon G\to\bar\dbU_{n+1}$ such that $\bar\rho_{i,i+1}=\alpha_i$ for all $i=1,\ldots,n$.
Now consider the restriction $$\bar\rho\vert_{G_0}\colon G_0\longrightarrow\bar\dbU_{n+1}.$$
Then again by Proposition~\ref{prop:masse unip}--(i) the $n$-fold Massey product $\langle\alpha_1\vert_{G_0},\ldots,\alpha_n\vert_{G_0}\rangle$ is defined in $\bfH^\bullet(G_0)$, too, and thus by hypothesis it vanishes.
Hence Proposition~\ref{prop:masse unip}--(ii) yields a homomorphism $\rho\colon G_0\to\dbU_{n+1}$ satisfying $\rho_{i,i+1}=\alpha_i\vert_{G_0}$ for all $i=1,\ldots,n$.
Then $\rho\circ\pi\colon G\to\dbU_{n+1}$ is a homomorphism satisfying 
$$(\rho\circ\pi)_{i,i+1}=\alpha_i\vert_{G_0}\circ\pi=\alpha_i\qquad\text{for every }i,$$ and by Proposition~\ref{prop:masse unip}--(ii) the $n$-fold Massey product $\langle\alpha_1,\ldots,\alpha_n\rangle$ vanishes in $\bfH^\bullet(G)$.
This proves (i) in case (a).

Now assume just that $\alpha_i\smallsmile\alpha_{i+1}=0$ for all $i=1,\ldots,n-1$, and that $\alpha_i\neq 0$ for every $i$ (cf. Proposition~\ref{prop:trivial strong}), to prove statement~(ii).
Since  
$$\alpha_i\vert_{G_0}\smallsmile\alpha_{i+1}\vert_{G_0}=\mathrm{res}_{G,G_0}^2(\alpha_i\smallsmile\alpha_{i+1})=\mathrm{res}_{G,G_0}^2(0)=0$$
(cf. \cite[Prop.~1.5.3]{nsw:cohn}) for all $i=1,\ldots,n-1$,
the $n$-fold Massey product $\langle\alpha_1\vert_{G_0},\ldots,\alpha_n\vert_{G_0}\rangle$ vanishes in $\bfH^\bullet(G_0)$ by hypothesis.
Hence, by Proposition~\ref{prop:masse unip} there exists a homomorphism $\rho\colon G_0\to\dbU_{n+1}$ such that $\rho_{i,i+1}=\alpha_i\vert_{G_0}$ for every $i=1,\ldots,n$.
Then $\rho\circ\pi\colon G\to\dbU_{n+1}$ is a homomorphism satisfying $(\rho\circ\pi)_{i,i+1}=\alpha_i$ for every $i=1,\ldots,n$, and by Proposition~\ref{prop:masse unip}--(ii) the $n$-fold Massey product $\langle\alpha_1,\ldots,\alpha_n\rangle$ vanishes in $\bfH^\bullet(G)$.
This proves (ii) in case~(a).

\medskip
\noindent {\sl Case (b).} 
Assume the $\alpha_i$'s are non-trivial multiples of each other.
Since $(G,\theta)$ is Kummerian, Theorem~\ref{thm:kummer Massey cyc} implies that the $n$-fold Massey product 
$\langle\alpha_1,\ldots,\alpha_1\rangle$ vanishes in $\bfH^\bullet(G)$.
Then by Proposition~\ref{prop:Massey cup}--(ii), one has
\[\begin{split}
 \langle\alpha_1,\alpha_2,\ldots,\alpha_n\rangle 
 &\supseteq \left\{\:\frac{b_2}{b_1}\cdot\beta\:\mid\:\beta\in \langle\alpha_1,\alpha_1,\alpha_3,\ldots,\alpha_n\rangle\:\right\}\\
 &\;\;\vdots\\
 &\supseteq
 \left\{\:\frac{b_2\cdots b_n}{b_1^{n-1}}\cdot\beta\:\mid\:\beta\in \langle\alpha_1,\ldots,\alpha_1\rangle\:\right\}\ni0,   
  \end{split}
\]
and thus $\langle\alpha_1,\ldots,\alpha_n\rangle$ vanishes in $\bfH^\bullet(G)$.
This proves both (i) and (ii) in case~(b).
\end{proof}

 %%%%%%%%%%%%%%%%%%%%%%

%%%%%%%%%%%%%%%%%%%%%%%%%%%%%%%%%%%%%%%%%%%%%%%%%%%%%%%%%%%%%%%%%%%%%%%%

From Theorem~\ref{thm:semidirectprod} we deduce the following.

\begin{cor}\label{cor:semidirect}
 Let $\K$ be a field containing a root of 1 of order $p$ {\rm(}and $\sqrt{-1}\in\K$, if $p=2${\rm)}.
 Then the maximal pro-$p$ Galois group $G_{\K}(p)$ of $\K$ satisfies the strong $n$-Massey vanishing property for every $n>2$ in the following cases:
 \begin{itemize}
  \item[(i)] $\K$ is a $p$-rigid field;
  \item[(ii)] $\K$ is a valued $p$-Henselian field whose residue field $\kappa$ has maximal pro-$p$ Galois group satisfying the strong $n$-Massey vanishing property for every $n>2$.  
 \end{itemize}
\end{cor}

\begin{proof}
 In the first case, the oriented pro-$p$ group $(G_{\K}(p),\theta_{\K})$ is $\theta_{\K}$-abelian by \cite[Cor.~3.17]{cmq:fast}.
 In the second case one has $G_{\K}(p)=A\rtimes_{\theta_{\kappa}} G_{\kappa}(p)$, with $A$ a free abelian pro-$p$ group, as shown in \cite[\S~1]{EK} if $\mathrm{char}(\kappa)\neq p$ (see also \cite[Thm.~3.6]{wadsworth}), and \cite[\S~3]{henselian} if $\mathrm{char}(\kappa)=p$.
\end{proof}

 %%%%%%%%%%%%%%%%%%%%%%

%%%%%%%%%%%%%%%%%%%%%%%%%%%%%%%%%%%%%%%%%%%%%%%%%%%%%%%%%%%%%%%%%%%%%%%%

\subsection{Proof of Theorem~1.2}
\label{ssec:proof}

We are ready to prove Theorem~\ref{thm:main}.

\begin{thm}\label{thm:Massey ETC}
 Let $(G,\theta)$ be an oriented pro-$p$ group of elementary type, and suppose that either $\theta$ is a torsion-free orientation. 
 Then $G$ satisfies the strong $n$-Massey vanishing property for every $n>2$.
\end{thm}

 \begin{proof}
  We proceed following the inductive construction of the oriented pro-$p$ group of elementary type.
  
  If $G$ is a free pro-$p$ group, then it is straightforward to see that $G$ satisfies the strong $n$-Massey vanishing property for every $n\geq0$ by Proposition~\ref{prop:masse unip}--(ii) (cf., e.g., \cite[Ex.~4.1]{mt:Massey}).
%  then $G$ is a projective pro-$p$ group (cf., e.g., \cite[Ch.~I, \S~5.9]{serre:galc}),
%  and for any sequence $\alpha_1,\ldots\alpha_n$ of non-trivial elements of $\rmH^1(G,\F_p)$, one may construct a suitable continuous homomorphism $\rho\colon G\to\dbU_{n+1}$.
%  Thus, by Proposition~\ref{prop:masse unip}--(ii) $G$ satisfies the strong $n$-Massey vanishing property for every $n\geq0$ (cf. e.g., \cite[Ex.~4.1]{mt:Massey}).

  If $G$ is a Demushkin group, then $G$ satisfies the strong $n$-Massey vanishing property as shown by A.~P\'al and E.~Szab\'o in \cite[Thm.~3.5]{pal:Massey} (see also \cite[Prop.~4.1]{JT:U4}).
 
 If $(G_1,\theta_1)$ and $(G_2,\theta_2)$ are two oriented pro-$p$ groups such that both $G_1$ and $G_2$ satisfy the strong $n$-Massey vanishing property, then also the free pro-$p$ product $G_1\ast G_2$ satisfies the strong $n$-Massey vanishing property (cf. \cite[Prop.~4.8]{JT:U4} and \cite[Rem.~5.2]{BCQ:AbsGalType}).
%Indeed, one has a decomposition $$\bfH^\bullet(G_1\ast G_2)\simeq\bfH^\bullet(G_1)\oplus\bfH^\bullet(G_2)$$
% (cf., e.g., \cite[Thm.~4.1.4--4.1.5]{nsw:cohn}).
% Thus, given a sequence  $\alpha_1,\ldots,\alpha_n$ of elements of $\rmH^1(G,\F_p)$, where $\alpha_i=\alpha_i\vert_{G_1}+\alpha_i\vert_{G_2}$ for $j=1,2$, the vanishing of the $n$-fold Massey product $\langle\alpha_n\vert_{G_j},\ldots,\alpha_n\vert_{G_j}\rangle$ gives rise to a continuous homomorphism $\rho_j\colon G_j\to\dbU_{n+1}$, cf. Proposition~\ref{prop:Massey cup}--(ii).
%The homomorphisms $\rho_1,\rho_2$ give rise, by the universal property of free pro-$p$ products (cf., e.g., \cite[Def.~4.1.1]{nsw:cohn}), to a continuous homomorphism 
%$$\rho\colon G_1\ast G_2\longrightarrow\dbU_{n+1}\qquad\text{s.t. }\rho_{i,i+1}=\alpha_i\;\forall\:i=1,\ldots,n,$$
%and one applies again Proposition~\ref{prop:Massey cup}--(ii).
 
 Finally, suppose that $(G_0,\theta)$ is an oriented pro-$p$ group of elementary type with $G_0$ satisfying the strong $n$-Massey vanishing property for every $n>2$, and consider the semidirect product $$(G,\tilde\theta)=(\Z_p\rtimes_{\theta} G_0,\tilde\theta).$$
Since $\theta$ is a torsion-free orientation, $(G,\tilde\theta)$ is Kummerian by Example~\ref{ex:kummer}--(d), and thus also $\Z_p\rtimes_{\theta} G_0$ satisfies the strong $n$-Massey vanishing property for every $n>2$ by Theorem~\ref{thm:semidirectprod}.
% On the other hand, if $p=2$ and $\Img(\theta)=\{\pm1\}$, then also $\Z_2\rtimes_{\theta} G_0$ satisfies the strong $n$-Massey vanishing property for every $n>2$ by Proposition~\ref{prop:semidirect Massey pm1}.
 \end{proof}
 
Items~(a)--(c) of Corollary~\ref{cor:fields} follow from Proposition~\ref{prop:etc fields} and Theorem~\ref{thm:Massey ETC}, 
and Items~(d)--(e) of Corollary~\ref{cor:fields} follow from Corollary~\ref{cor:semidirect}.  
%\begin{cor}
% Let $\K$ be a field containing a root of 1 of order $p$ {\rm(}assume further that either $\sqrt{-1}\in\K$ or $\K(\sqrt{-1})$ contains all roots of 1 of 2-power order, if $p=2${\rm)}, such that the quotient $\K^\times/(\K^\times)^p$ is finite. 
% Then the maximal pro-$p$ Galois $G_{\K}(p)$ satisfies the strong $n$-Massey vanishing property for every $n>2$ in the following cases:
% \begin{itemize}
%\item[(a)] $\K$ is a PAC field, or an extension of relative transcendence degree 1 of a PAC field;
%\item[(b)] $\K$ is an extension of transcendence degree 1 of a local field;
%\item[(c)] $\K$ is a valued $p$-Henselian field with residue field $\kappa$, such that $G_\kappa(p)$ satisfies the strong $n$-Massey vanishing property for every $n>2$ --- in particular, if $\K$ is $p$-rigid;
%\item[(d)] $\K$ is an algebraic extension of a global field of characteristic not $p$.
% \end{itemize}
%\end{cor}

%%%%%%%%%%%%%%%%%%%%%%%%%%%%%5

\begin{rem}\label{rem:subgroups}\rm
Let $(G,\theta)$ be an oriented pro-$p$ group of elementary type such that $\theta$ is a torsion-free orientation.
Since $(H,\theta\vert_H)$ is again an oriented pro-$p$ group of elementary type for every finitely generated subgroup $H\subseteq G$ (cf. Remark~\ref{rem:subgroups et}), Theorem~\ref{thm:Massey ETC} implies that every finitely generated subgroup of $G$ --- in particular, every {\sl open} subgroup (as an open subgroup of a finitely generated pro-$p$ group is again finitely generated; cf., e.g., \cite[Prop.~1.7]{ddsms}) --- satisfies the strong $n$-Massey vanishing property for every $n>2$.
\end{rem}


\begin{bibdiv}
\begin{biblist}


\bib{matrix1}{article}{
   author={Bier, A.},
   author={Ho\l ubowski, W.},
   title={A note on commutators in the group of infinite triangular matrices
   over a ring},
   journal={Linear Multilinear Algebra},
   volume={63},
   date={2015},
   number={11},
   pages={2301--2310},
}

\bib{BCQ:AbsGalType}{article}{
   author={Blumer, S.},
   author={Cassella, A.},
   author={Quadrelli, C.},
   title={Groups of $p$-absolute Galois type that are not absolute Galois
   groups},
   journal={J. Pure Appl. Algebra},
   volume={227},
   date={2023},
   number={4},
   pages={Paper No. 107262},
}

%\bib{BQW}{unpublished}{
%   author={Blumer, S.},
%   author={Quadrelli, C.},
%   author={Weigel, Th.S.},
%   title={Oriented right-angled Artin pro-$p$ groups and maximal pro-$p$ Galois groups},
%   date={2023},
%   note={Preprint, available at {\tt arXiv:2304.08123}},
%}	
   
\bib{cem}{article}{
   author={Chebolu, S.K.},
   author={Efrat, I.},
   author={Mina\v{c}, J.},
   title={Quotients of absolute Galois groups which determine the entire
   Galois cohomology},
   journal={Math. Ann.},
   volume={352},
   date={2012},
   number={1},
   pages={205--221},
   issn={0025-5831},
}

	
\bib{cmq:fast}{article}{
   author={Chebolu, S.K.},
   author={Mina\v{c}, J.},
   author={Quadrelli, C.},
   title={Detecting fast solvability of equations via small powerful Galois groups},
   journal={Trans. Amer. Math. Soc.},
   volume={367},
   number={21},
   date={2015},
   pages={8439--8464},
}


	
\bib{ddsms}{book}{
   author={Dixon, J.D.},
   author={du Sautoy, M.P.F.},
   author={Mann, A.},
   author={Segal, D.},
   title={Analytic pro-$p$ groups},
   series={Cambridge Studies in Advanced Mathematics},
   volume={61},
   edition={2},
   publisher={Cambridge University Press, Cambridge},
   date={1999},
   pages={xviii+368},
   isbn={0-521-65011-9},
}

\bib{dwyer}{article}{
   author={Dwyer, W.G.},
   title={Homology, Massey products and maps between groups},
   journal={J. Pure Appl. Algebra},
   volume={6},
   date={1975},
   number={2},
   pages={177--190},
   issn={0022-4049},
}

\bib{ido:etc}{article}{
   author={Efrat, I.},
   title={Orderings, valuations, and free products of Galois groups},
   journal={Sem. Structure Alg\'ebriques Ordonn\'ees, Univ. Paris VII},
   date={1995},
}

\bib{ido:etc2}{article}{
   author={Efrat, I.},
   title={Pro-$p$ Galois groups of algebraic extensions of $\mathbf{Q}$},
   journal={J. Number Theory},
   volume={64},
   date={1997},
   number={1},
   pages={84--99},
}

\bib{efrat:small}{article}{
   author={Efrat, I.},
   title={Small maximal pro-$p$ Galois groups},
   journal={Manuscripta Math.},
   volume={95},
   date={1998},
   number={2},
   pages={237--249},
   issn={0025-2611},
}

\bib{henselian}{article}{
   author={Efrat, I.},
   title={Finitely generated pro-$p$ Galois groups of $p$-Henselian fields},
   journal={J. Pure Appl. Algebra},
   volume={138},
   date={1999},
   number={3},
   pages={215--228},
}

\bib{efrat:funfield}{article}{
   author={Efrat, I.},
   title={Pro-$p$ Galois groups of function fields over local fields},
   journal={Comm. Algebra},
   volume={28},
   date={2000},
   number={6},
   pages={2999--3021},
   issn={0092-7872},}

\bib{ido:Hasse}{article}{
   author={Efrat, I.},
   title={A Hasse principle for function fields over PAC fields},
   journal={Israel J. Math.},
   volume={122},
   date={2001},
   pages={43--60},
   issn={0021-2172},
}

\bib{ido:Massey}{article}{
   author={Efrat, I.},
   title={The Zassenhaus filtration, Massey products, and homomorphisms of
   profinite groups},
   journal={Adv. Math.},
   volume={263},
   date={2014},
   pages={389--411},
}

\bib{idomatrix}{article}{
   author={Efrat, I.},
   title={The lower $p$-central series of a free profinite group and the
   shuffle algebra},
   journal={J. Pure Appl. Algebra},
   volume={224},
   date={2020},
   number={6},
   pages={106260, 13},
   issn={0022-4049},
}

\bib{EM:Massey}{article}{
   author={Efrat, I.},
   author={Matzri, E.},
   title={Triple Massey products and absolute Galois groups},
   journal={J. Eur. Math. Soc. (JEMS)},
   volume={19},
   date={2017},
   number={12},
   pages={3629--3640},
   issn={1435-9855},
}

%\bib{idojan:descending}{article}{
%   author={Efrat, I.},
%   author={Mina\v{c}, J.},
%   title={On the descending central sequence of absolute Galois groups},
%   journal={Amer. J. Math.},
%   volume={133},
%   date={2011},
%   number={6},
%   pages={1503--1532},
%}
	\bib{idojan:tams}{article}{
   author={Efrat, I.},
   author={Mina\v{c}, J.},
   title={Galois groups and cohomological functors},
   journal={Trans. Amer. Math. Soc.},
   volume={369},
   date={2017},
   number={4},
   pages={2697--2720},
   issn={0002-9947},
}

\bib{eq:kummer}{article}{
   author={Efrat, I.},
   author={Quadrelli, C.},
   title={The Kummerian property and maximal pro-$p$ Galois groups},
   journal={J. Algebra},
   volume={525},
   date={2019},
   pages={284--310},
   issn={0021-8693},
}

\bib{EK}{article}{
   author={Engler, A.J.},
   author={Koenigsmann, J.},
   title={Abelian subgroups of pro-$p$ Galois groups},
   journal={Trans. Amer. Math. Soc.},
   volume={350},
   date={1998},
   number={6},
   pages={2473--2485},
   issn={0002-9947},
}

\bib{friedjarden}{book}{
   author={Fried, M. D.},
   author={Jarden, M.},
   title={Field arithmetic},
   series={Ergebnisse der Mathematik und ihrer Grenzgebiete. 3. Folge. A
   Series of Modern Surveys in Mathematics },
   volume={11},
   edition={3},
   note={Revised by Jarden},
   publisher={Springer-Verlag, Berlin},
  date={2008},
   pages={xxiv+792},
}

\bib{jochen}{article}{
   author={G\"{a}rtner, J.},
   title={Higher Massey products in the cohomology of mild pro-$p$-groups},
   journal={J. Algebra},
   volume={422},
   date={2015},
   pages={788--820},
   issn={0021-8693},
}

\bib{WDG}{article}{
   author={Geyer, W.-D.},
   title={Field theory},
   conference={
      title={Travaux math\'{e}matiques. Vol. XXII},
   },
   book={
      series={Trav. Math.},
      volume={22},
      publisher={Fac. Sci. Technol. Commun. Univ. Luxemb., Luxembourg},
   },
   date={2013},
   pages={5--177},
}

\bib{GPM}{article}{
   author={Guillot, P.},
   author={Mina\v{c}, J.},
   author={Topaz, A.},
   title={Four-fold Massey products in Galois cohomology},
   note={With an appendix by O.~Wittenberg},
   journal={Compos. Math.},
   volume={154},
   date={2018},
   number={9},
   pages={1921--1959},
}


	\bib{PJ}{article}{
   author={Guillot, P.},
   author={Mina\v{c}, J.},
   title={Extensions of unipotent groups, Massey products and Galois theory},
   journal={Adv. Math.},
   volume={354},
   date={2019},
   pages={article no. 106748},
}
	
	\bib{HW:book}{book}{
   author={Haesemeyer, C.},
   author={Weibel, Ch.},
   title={The norm residue theorem in motivic cohomology},
   series={Annals of Mathematics Studies},
   volume={200},
   publisher={Princeton University Press, Princeton, NJ},
   date={2019},
}

\bib{HW:Massey}{article}{
   author={Harpaz, Y.},
   author={Wittenberg, O.},
   title={The Massey vanishing conjecture for number fields},
   journal={Duke Math. J.},
   volume={172},
   date={2023},
   number={1},
   pages={1--41},
}

\bib{hopwick}{article}{
   author={Hopkins, M.J.},
   author={Wickelgren, K.G.},
   title={Splitting varieties for triple Massey products},
   journal={J. Pure Appl. Algebra},
   volume={219},
   date={2015},
   number={5},
   pages={1304--1319},
}

\bib{jacobware}{article}{
   author={Jacob, B.},
   author={Ware, R.},
   title={A recursive description of the maximal pro-$2$ Galois group via
   Witt rings},
   journal={Math. Z.},
   volume={200},
   date={1989},
   number={3},
   pages={379--396},
   issn={0025-5874},}
   
\bib{Koenigsmann}{article}{
   author={Koenigsmann, J.},
   title={Pro-$p$ Galois groups of rank $\leq 4$},
   journal={Manuscripta Math.},
   volume={95},
   date={1998},
   number={2},
   pages={251--271},
   issn={0025-2611},}

	\bib{kraines}{article}{
   author={Kraines, D.},
   title={Massey higher products},
   journal={Trans. Amer. Math. Soc.},
   volume={124},
   date={1966},
   pages={431--449},
   issn={0002-9947},
}
	
\bib{labute}{article}{
   author={Labute, J.P.},
   title={Classification of Demushkin groups},
   journal={Canadian J. Math.},
   volume={19},
   date={1967},
   pages={106--132},
   issn={0008-414X},
}
		

\bib{LLSWW}{article}{
   author={Lam, Y.H.J.},
   author={Liu, Y.},
   author={Sharifi, R.T.},
   author={Wake, P.},
   author={Wang, J.},
   title={Generalized Bockstein maps and Massey products},
   date={2023},
   journal={Forum Math. Sigma},
   volume={11},
   date={2023},
   pages={Paper No. e5},
}


\bib{marshall}{article}{
   author={Marshall, M.},
   title={The elementary type conjecture in quadratic form theory},
   conference={
      title={Algebraic and arithmetic theory of quadratic forms},
   },
   book={
      series={Contemp. Math.},
      volume={344},
      publisher={Amer. Math. Soc., Providence, RI},
   },
   date={2004},
   pages={275--293},}

\bib{eli:Massey}{unpublished}{
   author={Matzri, E.},
   title={Triple Massey products in Galois cohomology},
   date={2014},
   note={Preprint, available at {\tt arXiv:1411.4146}},
}

\bib{eli2}{article}{
   author={Matzri, E.},
   title={Triple Massey products of weight $(1,n,1)$ in Galois cohomology},
   journal={J. Algebra},
   volume={499},
   date={2018},
   pages={272--280},
}
\bib{eli3}{article}{
   author={Matzri, E.},
   title={Higher triple Massey products and symbols},
   journal={J. Algebra},
   volume={527},
   date={2019},
   pages={136--146},
   issn={0021-8693},}
		
	

\bib{MerSca1}{unpublished}{
   author={Merkurjev, A.},
   author={Scavia, F.},
   title={Degenerate fourfold Massey products over arbitrary fields},
   date={2022},
   note={Preprint, available at {\tt arXiv:2208.13011}},
}

\bib{MerSca2}{unpublished}{
   author={Merkurjev, A.},
   author={Scavia, F.},
   title={The Massey Vanishing Conjecture for fourfold Massey products modulo 2},
   date={2023},
   note={Preprint, available at {\tt arXiv:2301.09290}},
}


\bib{MerSca3}{unpublished}{
   author={Merkurjev, A.},
   author={Scavia, F.},
   title={On the Massey Vanishing Conjecture and Formal Hilbert 90},
   date={2023},
   note={Preprint, available at {\tt arXiv:2308.13682}},
}

%\bib{minac:pyt}{article}{
%   author={Mina\v{c}, J.},
%   title={Galois groups of some $2$-extensions of ordered fields},
%   journal={C. R. Math. Rep. Acad. Sci. Canada},
%   volume={8},
%   date={1986},
%   number={2},
%   pages={103--108},
%}
		
\bib{MPQT}{article}{
   author={Mina\v{c}, J.},
   author={Pasini, F.W.},
   author={Quadrelli, C.},
   author={T\^{a}n, N.D.},
   title={Koszul algebras and quadratic duals in Galois cohomology},
   journal={Adv. Math.},
   volume={380},
   date={2021},
   pages={article no. 107569},
}
	
		
   
\bib{birs}{report}{
   author={Mina\v{c}, J.},
   author={Pop, F.},
   author={Topaz, A.},
   author={Wickelgren, K.},
   title={Nilpotent Fundamental Groups},
   date={2017},
   note={Report of the workshop ``Nilpotent Fundamental Groups'', Banff AB, Canada, June 2017},
   eprint={https://www.birs.ca/workshops/2017/17w5112/report17w5112.pdf},
   organization={BIRS for Mathematical Innovation and Discovery},
   conference={
      title={Nilpotent Fundamental Groups 17w5112},
      address={Banff AB, Canada},
      date={June 2017}},
}



   \bib{mt:conj}{article}{
   author={Mina\v{c}, J.},
   author={T\^{a}n, N.D.},
   title={The kernel unipotent conjecture and the vanishing of Massey
   products for odd rigid fields},
   journal={Adv. Math.},
   volume={273},
   date={2015},
   pages={242--270},
}

%\bib{mt:docu}{article}{
%   author={Mina\v{c}, J.},
%   author={T\^{a}n, N.D.},
%   title={Triple Massey products over global fields},
%   journal={Doc. Math.},
%   volume={20},
%   date={2015},
%   pages={1467--1480},
%}
		
     
\bib{MT:Masseyall}{article}{
   author={Mina\v{c}, J.},
   author={T\^{a}n, N.D.},
   title={Triple Massey products vanish over all fields},
   journal={J. London Math. Soc.},
   volume={94},
   date={2016},
   pages={909--932},
}


   
   \bib{JT:U4}{article}{
   author={Mina\v{c}, J.},
   author={T\^{a}n, N.D.},
   title={Counting Galois $\dbU_4(\F_p)$-extensions using Massey products},
   journal={J. Number Theory},
   volume={176},
   date={2017},
   pages={76--112},
   issn={0022-314X},
}

\bib{mt:Massey}{article}{
   author={Mina\v{c}, J.},
   author={T\^{a}n, N.D.},
   title={Triple Massey products and Galois theory},
   journal={J. Eur. Math. Soc. (JEMS)},
   volume={19},
   date={2017},
   number={1},
   pages={255--284},
   issn={1435-9855},
}

\bib{nsw:cohn}{book}{
   author={Neukirch, J.},
   author={Schmidt, A.},
   author={Wingberg, K.},
   title={Cohomology of number fields},
   series={Grundlehren der Mathematischen Wissenschaften},
   volume={323},
   edition={2},
   publisher={Springer-Verlag, Berlin},
   date={2008},
   pages={xvi+825},
   isbn={978-3-540-37888-4},}
  

\bib{palquick:Massey}{unpublished}{
   author={P\'al, A.},
   author={Quick, G.},
   title={Real projective groups are formal},
   date={2022},
   note={Preprint, available at {\tt arXiv:2206.14645}},
}
  
\bib{pal:Massey}{unpublished}{
   author={P\'al, A.},
   author={Szab\'o, E.},
   title={The strong Massey vanishing conjecture for fields with virtual cohomological dimension at most 1},
   date={2020},
   note={Preprint, available at {\tt arXiv:1811.06192}},
}
  
%  \bib{posi:formal}{article}{
%   author={Positselski, L.},
%   title={Koszulity of cohomology = $K(\pi,1)$-ness + quasi-formality},
%   journal={J. Algebra},
%   volume={483},
%   date={2017},
%   pages={188--229},
%   issn={0021-8693},
%}
  

   
\bib{cq:bk}{article}{
   author={Quadrelli, C.},
   title={Bloch-Kato pro-$p$ groups and locally powerful groups},
   journal={Forum Math.},
   volume={26},
   date={2014},
   number={3},
   pages={793--814},
   issn={0933-7741},
}

%\bib{cq:chase}{unpublished}{
%   author={Quadrelli, C.},
%   title={Chasing maximal pro-$p$ Galois groups via 1-cyclotomicity},
%   date={2021},
%   note={Preprint, available at {\tt arXiv:2106.00335}},
%}
  
\bib{qw:cyc}{article}{
   author={Quadrelli, C.},
   author={Weigel, Th.S.},
   title={Profinite groups with a cyclotomic $p$-orientation},
   date={2020},
   volume={25},
   journal={Doc. Math.},
   pages={1881--1916}
   }

	
  
\bib{qw:bogo}{article}{
   author={Quadrelli, C.},
   author={Weigel, Th.S.},
   title={Oriented pro-$\ell$ groups with the Bogomolov-Positselski property},
   date={2022},
   volume={8},
   journal={Res. Number Theory},
   number={2},
   pages={article no. 21},
   }

	

	
\bib{rost}{article}{
   author={Rost, M.},
   title={Norm varieties and algebraic cobordism},
   conference={
      title={Proceedings of the International Congress of Mathematicians,
      Vol. II},
      address={Beijing},
      date={2002},
   },
   book={
      publisher={Higher Ed. Press, Beijing},
   },
   date={2002},
   pages={77--85},}
   
   \bib{serre:Demushkin}{article}{
   author={Serre, J-P.},
   title={Structure de certains pro-$p$-groupes (d'apr\`es Demu\v{s}kin)},
   language={French},
   conference={
      title={S\'{e}minaire Bourbaki, Vol. 8},
   },
   book={
      publisher={Soc. Math. France, Paris},
   },
   date={1995},
   pages={Exp. No. 252, 145--155},
}
		
\bib{serre:galc}{book}{
   author={Serre, J-P.},
   title={Galois cohomology},
   series={Springer Monographs in Mathematics},
   edition={Corrected reprint of the 1997 English edition},
   note={Translated from the French by Patrick Ion and revised by the
   author},
   publisher={Springer-Verlag, Berlin},
   date={2002},
   pages={x+210},
   isbn={3-540-42192-0},}
	
	

	


\bib{sz:raags}{article}{
	author={Snopce, I.},
	author={Zalesski\u{\i}, P.A.},
	title={Right-angled Artin pro-$p$ groups},
	date={2022},
	journal={Bull. Lond. Math. Soc.},
	volume={54},
	pages={1904-1922},
	number={5},
}
	

\bib{voev}{article}{
   author={Voevodsky, V.},
   title={On motivic cohomology with $\mathbf{Z}/l$-coefficients},
   journal={Ann. of Math. (2)},
   volume={174},
   date={2011},
   number={1},
   pages={401--438},
   issn={0003-486X},
   }
   
   
   \bib{vogel}{report}{
   author={Vogel, D.},
   title={Massey products in the Galois cohomology of number fields},
   date={2004},
   note={PhD thesis, University of Heidelberg},
   eprint={http://www.ub.uni-heidelberg.de/archiv/4418},
}
	
%   
%\bib{weibel}{article}{
%   author={Weibel, Ch.},
%   title={2007 Trieste lectures on the proof of the Bloch-Kato conjecture},
%   conference={
%      title={Some recent developments in algebraic $K$-theory},
%   },
%   book={
%      series={ICTP Lect. Notes},
%      volume={23},
%      publisher={Abdus Salam Int. Cent. Theoret. Phys., Trieste},
%   },
%   date={2008},
%   pages={277--305},
%}



\bib{wadsworth}{article}{
   author={Wadsworth, A.},
   title={$p$-Henselian field: $K$-theory, Galois cohomology, and graded
   Witt rings},
   journal={Pacific J. Math.},
   volume={105},
   date={1983},
   number={2},
   pages={473--496},
}
		

\bib{ware}{article}{
   author={Ware, R.},
   title={Galois groups of maximal $p$-extensions},
   journal={Trans. Amer. Math. Soc.},
   volume={333},
   date={1992},
   number={2},
   pages={721--728},
}

\bib{weibel}{article}{
   author={Weibel, Ch.},
   title={The norm residue isomorphism theorem},
   journal={J. Topol.},
   volume={2},
   date={2009},
   number={2},
   pages={346--372},
   issn={1753-8416},
}


\bib{wick}{article}{
   author={Wickelgren, K.},
   title={Massey products $\langle y$, $x$, $x$, $\ldots$, $x$, $x$,
   $y\rangle$ in Galois cohomology via rational points},
   journal={J. Pure Appl. Algebra},
   volume={221},
   date={2017},
   number={7},
   pages={1845--1866},
   issn={0022-4049},
}
	

\end{biblist}
\end{bibdiv}
\end{document}